\documentclass{amsart}[draft]
\usepackage{preambleFL}
\usepackage{graphicx} 
\usepackage[all]{xy}
\usepackage{enumerate}

\newcommand{\mI}{\mathcal{I}}

\newcommand{\sB}{\mathscr{B}}
\newcommand{\sD}{\mathscr{D}}
\newcommand{\sE}{\mathscr{E}}
\newcommand{\sF}{\mathscr{F}}
\newcommand{\sV}{\mathscr{V}}
\newcommand{\sI}{\mathscr{I}}

\newcommand{\sW}{\mathscr{W}}

\newcommand{\bD}{\mathbb{D}}

\newcommand{\bj}{\mathbf{j}}
\newcommand{\bk}{\mathbf{k}}
\newcommand{\bm}{\mathbf{m}}
\newcommand{\bn}{\mathbf{n}}

\newcommand{\bt}{\mathbf{t}}

\newcommand{\Ek}{\sE^{\boxtimes k}}

\newcommand{\rtarr}{\to}

\newcommand{\comm}{\ensuremath{\mathscr{C\mkern-3mu}omm}}
\newcommand{\asso}{\ensuremath{\mathscr{A\mkern-3mu}ss}}
\newcommand{\fin}{\ensuremath{\mathscr{F}_{>0}}}
\newcommand{\fins}{\ensuremath{\mathscr{F}}}

\newcommand{\SI}{\Sigma}
\newcommand{\si}{\sigma}
\newcommand{\LA}{\Lambda}

\newcommand{\letter}{x}
\newcommand{\e}{X}

\newcommand{\Rmod}{\mathrm{RMod}}
\newcommand{\Lmod}{\mathrm{LMod}}
\newcommand{\Mon}{\mathrm{Mon}}

\title{On the Kelly monoidal structure of $\Lambda$-sequences and unital operads}
\author{Aowen Fan and Foling Zou}
\date{}

\begin{document}
\begin{abstract}
  Let $\Lambda$ be the category of based finite sets $\bn$ and based injections. We study properties of monoids and modules in
  $\Lambda$-sequences under the Kelly monoidal structure. In particular, we show that
  the forgetful functor from right modules in $\Lambda$-sequences to right modules in
  symmetric sequences is an isomorphism. We show that any compatible lower data
  extends to a normal oplax monoidal structure and use this to establish a
  universal normal oplax monoidal structure on $\Lambda$-sequences extending the
  Kelly product, identifying unital operads to monoids in unital $\Lambda$-sequences
  for a general symmetric monoidal category $\sV$. We also establish a closed
  monoidal localization theorem.
\end{abstract}
\maketitle

\tableofcontents

\section*{Introduction}
The notion of operads was introduced by Peter May in the 1970s \cite{MayGeo}, and has since been
prevalent in studying algebraic structures in the context of topological spaces,
spectra, algebras and more. It was first observed by Kelly \cite{Kelly0} that the structure of an operad in a
closed symmetric monoidal category $\sV$ is the same as a monoid in symmetric
sequences in $\sV$, under a product $\odot_{\Sigma}$ he defined. This product is also
called the composition product (see \cite{MZZ} for a historical account). This perspective
is convenient for studying model structures on operads and
their algebras, or for forming bar constructions \cite{BMAx, BergMoerOp2, Kro,
  MZZ, PS}. 

An operad $\mathscr{C}$ in a symmetric monoidal category
$(\sV, \otimes, \mI)$ is called \emph{based} if it comes with a map $\eta:\mI \to \sC(0)$
and \emph{unital} if $\mathscr{C}(0) = \mI$.
Unital operads are relevant in the context of based spaces, spectra or unital
$k$-algebras, as they build in the pre-assigned based condition of algebras into
the combinatorics of the operad. In \cite{MZZ}, May, Zhang and the second named
author developed a unital version of Kelly's theorem: replacing symmetric
sequences by $\Lambda$-sequences, unital (resp.\ based) operads in $\sV$ can
be identified with monoids in unital $\Lambda$-sequences (resp.\ $\Lambda$-sequences)  in $\sV$ 
under the Kelly product $\odot_{\Lambda}$ (see \autoref{sec:prelim} for definitions).

\medskip
\textbf{Modules.}
The identification of operads as monoids opens up the perspective of left,
  right modules and bar constructions. In \autoref{sec:modules}, we study
  properties of monoids and modules in $\Sigma$-objects and $\Lambda$-objects in a closed
  symmetric monoidal category $\sV$.
  In particular, we show the following result in \autoref{pullback}.
\begin{thm}\label{pullback-intro}
  Let $\sC$ be a based operad.  Then the forgetful map
  from right $\sC$-modules in $\Lambda$-objects to right $\sC$-modules in
  $\Sigma$-objects is an isomorphism.
\end{thm}
In particular (\autoref{ex-I1-comm}), taking $\sC = \sI_1$ where $\sI_1(\mathbf{0})= \sI_1(\mathbf{1})=\mI$ and $\sI_1(\bn)=\emptyset$ for
$n\geq2$, we recover \cite[5.1.8]{Fresse0}; taking $\sC$ to be the commutativity
operad, we cover \cite[4.10]{MZZ} and \cite[5.1.6]{Fresse0}.

\smallskip
We also found an interesting construction of new operads. For an operad $\sC$ in
a Cartesian monoidal category $\sV$, we construct a new operad $\sB$ with explicit
terms
\begin{align*}
  \sB(0) & = \sC(0), \\
  \sB(1) & = \sC(0) \sqcup \sC(1), \\
 \text{ and } \sB(n) & = \sqcup_{S \subset \{1,2,\cdots,n\}} \sC(|S|).
\end{align*}
The category of $\sB$-algebras seems ridiculously rich: each $\sB$-algebra is
simultaneously a $\sC$-algebra and a commutative algebra; both
$\sC$-algebras and commutative algebras are faithful subcategories of $\sB$-algebras
(see \autoref{prop:base-operad} and remarks after).

\smallskip
From an operad $\sC$ in $\sV$, two associated $\sV$-enriched symmetric monoidal
categories have been studied in the literature: the permutative envelop
$\overline{\sC}$ and the category of operators $\widehat{\sC}$.
They are constructed in explicit terms. In \cite{MZZ}, it is shown how to
construct $\overline{\sC}$ using the Day convolution and Kelly product;
in \autoref{prop:C-hat-comm}, we show how to construct
$\widehat{\sC}$ using the Day convolution and Kelly product.

The permutative envelop $\overline{\sC}$ has the property that the category of $\sC$-algebras in
$\sV$ id isomorphic to strong symmetric monoidal functors from $\overline{\sC}$
to $\sV$. When $\sV$ is Cartesian monoidal, this is also isomorphic to strong
symmetric monoidal functors from $\widehat{\sC}$ to $\sV$.
However, in homotopy theory it is common to replace the strong symmetric
monoidal condition with the Segal conditions, then $\widehat{\sC}$ becomes more
interesting.
For a subcategory $\Pi \subset \widehat{\sC}$, May--Thomason \cite{MT} uses
a monad  $\widehat{\bC}^{op}_{\Pi}$ on $\Pi[ \mathrm{Top}_{*}]$ to unify Segal's
$\fins$-space approach and May's operadic approach to infinite loop space
machines. Here, $\Pi[ \mathrm{Top}_{*}]$ is the category of covariant functors
from $\Pi$ to based spaces.

We show in \autoref{lem:monad-same} that for $\Sigma \subset \Lambda \subset \overline{\sC}$, the monads
$\overline{\bC}^{op}_{\Sigma}$ and $\overline{\bC}^{op}_{\Lambda}$ are isomorphic to the
monads associated to the monoid $\sC$ in $\Sigma$-objects and $\Lambda$-objects. With
this we show the following result in \autoref{thm:iso} for both $\Sigma$-objects and
$\Lambda$-objects, which gives a complete understanding of right $\sC$-modules.
Note that $\overline{\sC}^{op}$ does not depend on $\Sigma$ or $\Lambda$, so
this also gives another proof of \autoref{pullback-intro}.
\begin{thm}
The category of right $\sC$-modules is isomorphic to the category of $\sV$-enriched functors from $\overline{\sC}^{op}$ to $\sV$.
\end{thm}

\medskip
\textbf{Unital operads as monoids.}
  In order for the Kelly product to be a monoidal product, one uses that the
monoidal structure on $\sV$ is closed and in particular that $\otimes$ commutes with
small colimits. Sometimes one would like to drop this assumption.
For example, $\mathrm{Top}^{op}$ is not closed monoidal and the operads in it are
cooperads in $\mathrm{Top}$.

In \cite[Definitions 2.10-2.12]{ChingComposition}, Ching
generalized the notion of a monoidal structure on a category $\sE$ to a so-called normal oplax monoidal
structure, namely  for all $n \ge 2$ functors
  $$\mu_n : \sE^n \rtarr \sE,$$ and some compatible natural transformations
  $\alpha_{n,l,r}$.
  He then constructed a  normal oplax monoidal structure on symmetric sequences
  in $\sV$ where $\mu_2$ is the Kelly product.
  This way, he could still make sense of monoids and bar constructions.
  In fact, this approach can be generalized to $\Lambda$-sequences (\autoref{rem:Ching}), with
  significant modification due to the base point identifications.
  
  In \autoref{sec:normal-op-lax}, we offer an alternative approach.
  In \autoref{thm:complete-mu}, we show that any compatible data of
  $\mu_n$ and $\alpha_{n,l,r}$ for $n\leq 3$ extends to normal oplax structures, and
  there is a universal one. 
  Via our monoidal localization, this gives the result (\autoref{mainthm}):
  \begin{thm}\label{mainthm-intro}
    There is a universal normal oplax monoidal structure on $\Lambda$-sequences
    where $\mu_2$ and $\mu_3$ are the 2-fold and 3-fold Kelly products.
  \end{thm}
  In particular, the adaption of Ching's normal oplax monoidal structure has a
  morphism to our construction. Moreover, all sensible
  normal oplax monoidal structures are equivalent, having the same category of
  monoids~(\autoref{rem:all-same}). Using this structure, we establish the following result in
  \autoref{thm:operad}, answering a question left open in \cite{MZZ}.
\begin{thm}
  Let $\sV$ be a complete and cocomplete symmetric monoidal category.
  There are isomorphisms between categories of unital operads
  in $\sV$  and of monoids in unital $\Lambda$-sequences  in $\sV$;
  between categories of based operads in $\sV$ and of monoids in $\Lambda$-sequences in $\sV$.
\end{thm}

\medskip
\textbf{Monoidal localization.}
We study the localization of categories in \autoref{sec:mono-local}.
As a key step of our proof of \autoref{mainthm-intro}, we establish the
following \emph{closed monoidal localization theorem}
(\autoref{commutewithtensor}), which could be of independent interest. 
\begin{thm} \label{thm:monoidal-loc-intro}
  Let $\sV$ be a complete and cocomplete symmetric monoidal category. There
  exists a localization $L: \sV \to L\sV = \sV[\overline{S}^{-1}]$ such that 
\begin{itemize}
\item $L\sV$ is complete, cocomplete and symmetric monoidal;
\item $L$ preserves small limits, colimits and $\otimes$;
\item $\otimes$ commutes with small colimits in $L\sV$.
\end{itemize}
\end{thm}
\begin{rem}
  If $L\sV$ is presentable, it is closed monoidal by the adjoint functor
  theorem.
\end{rem}
  Day showed in \cite[Proposition 2.5]{day1973note} that any monoidal category
  can be embedded into a complete and cocomplete category in a way that
  preserves both the monoidal structure, limits and those colimits \emph{that are
    respected by the tensor product}. 
  The goal of his localization is for cocompleteness
  while the goal of our localization is to commute $\otimes$ with colimits.
  
  In the same paper, Day introduced the concept of a \emph{monoidal} collection of morphisms $S$ in
  $\sV$ such that the localization $L: \sV \to \sV[S^{-1}]$
  preserves the tensor product. We introduce the concept of $S$ being a
  \emph{$\kappa$-closed monoidal multiplicative system}, so that the localization $L: \sV \to \sV[S^{-1}]$
  preserves the tensor product as well as $\kappa$-small limits and colimits.
  We then show in \autoref{thm:closure} that   
\begin{thm}
  Let $\kappa$ be a regular cardinal.

Any collection of morphisms $S$ has a $\kappa$-closed
monoidal multiplicative closure $\overline{S}$.
\end{thm}

\textbf{Acknowledgments.}
The authors would like to thank Peter May and Jonathan Rubin for helpful conversations. 

\section{Preliminaries}
\label{sec:prelim}
We recall some definitions and results from \cite{MZZ}.
Let $\sV$ be a complete and cocomplete symmetric monoidal category throughout.
The category $\sV$ has three distinguished objects:
\begin{enumerate}
\item $\emptyset$, the initial object, that is, the coproduct of the empty set of objects.
\item $\ast$, the terminal object, that is, the product of the empty set of objects.
\item $\mI$, the unit object for $\otimes$ on $\sV$.
\end{enumerate}

\begin{defn}
  The category $\LA$ has objects of based finite sets $\mathbf{n} =\{0,1, \cdots, n\}$,
with base point $0$, and morphisms of based injections. The category $\SI$ is the
subcategory of $\LA$ with the same objects and isomorphisms, or permutations.
Note that morphisms in $\LA$ are generated by the permutations and ordered
injections  $\si_i\colon \bf{n-1}\rtarr \bf{n}$, $1\leq i\leq n$, that skip $i$ in
the target.
\end{defn}
A \emph{symmetric sequence} (also called \emph{$\Sigma$-sequence} or \emph{$\Sigma$-object}) in $\sV$ is a functor $\SI^{op}\rtarr \sV$;
a \emph{$\Lambda$-object} in $\sV$ is a functor $\sD: \Lambda^{op} \to \sV$; a \emph{(unitary) $\Lambda$-sequence}
in $\sV$ is a $\Lambda$-object $\sD$ together with a base map $\eta\colon
\mI \rtarr \sD (\mathbf 0)$. We say that a $\LA$-sequence $\sD$ is \emph{unital}
if $\sD (\mathbf{0})=\mI$ and $\eta$ is the identity map.
\begin{notn}\label{notn2}  We write
$\Sigma^{op}[\sV]$, $\LA^{op}[\sV]$, $\LA^{op}[\sV]_{\sI_0}$ and $\LA^{op}_{\mI}[\sV]$
for the category of $\SI$-sequences, $\LA$-objects,
$\LA$-sequences and unital $\LA$-sequences, where the morphisms are
natural transformations (that restrict to morphisms under $\mI$ at
$\mathbf{0}$). We have an inclusion of categories  $\LA^{op}_{\mI}[\sV]\subset
\LA^{op}[\sV]_{\sI_{0}}$. 
\end{notn}

\begin{defn}
  For $m \geq 0$, we let $\sI_{m}$ denote the $\LA$-sequence
  $$\sI_{m} =  \Lambda(-,\bm) \otimes \mI.$$
In particular,  $\sI_0(\mathbf{0})=\mI$ and $\sI_0(\bn)=\emptyset$ for $n\geq1$;
$\sI_1(\mathbf{0})= \sI_1(\mathbf{1})=\mI$ and $\sI_1(\bn)=\emptyset$ for $n\geq2$.
We abuse notation to also write the base map as $\eta: \sI_0 \to \sD$.
\end{defn}

\medskip
For $\sD$ and $\sE$ in  $\LA^{op}[\sV]$,
the Day convolution $\sD\boxtimes \sE$ is the left Kan extension
displayed in the diagram
\begin{equation}\label{Kan1}
  \begin{tikzcd}
     \LA^{op}\times \LA^{op} \ar[r,"{\sD \times \sE}"] \ar[d,"{\bigvee}"'] & \sV \times
     \sV \ar[r,"\otimes"]& \sV.\\
\LA^{op} \ar[urr,"{\sD\boxtimes \sE}"']  & \\
  \end{tikzcd}
\end{equation}
It can be computed by a coend
\begin{equation}\label{coend1}    
(\sD\boxtimes \sE)(\bn)  = \int^{\bj_1,\bj_2} \Lambda(\bn, \mathbf{j_1 +  j_2}) \otimes \big( \sD(\bj_1) \otimes \sE(\bj_2) \big).
\end{equation}

As proved in \cite[2.1]{Kelly}, it is formal that $\boxtimes$ is a symmetric monoidal
product which is closed when $\sV$ is and that $\sI_0$ is the unit for the Day
convolution.

When  $\sD$ and $\sE$ are $\Lambda$-sequences, so is $\sD \boxtimes \sE$.
The base maps of $\sD$ and $\sE$ induce the base map of $\sD \boxtimes \sE$ by
\begin{equation*}
  \begin{tikzcd}
    \eta: \sI_0 \cong \sI_0 \boxtimes \sI_0 \ar[r,"\eta \boxtimes \eta"] & \sD \boxtimes \sE.
  \end{tikzcd}
\end{equation*}

One can perform the Day convolution on $\Sigma^{op}[\sV]$ similarly, and we write
 $\boxtimes_{\Lambda}$  and $\boxtimes_{\Sigma}$ to indicate the context. For $\sD$ and $\sE$ in
$\LA^{op}[\sV]$, there is a natural map of $\SI$-sequences
$q \colon  \sD\boxtimes_{\SI} \sE \rtarr  \sD\boxtimes_{\LA} \sE$, which turns out to be  an
isomorphism of $\SI$-sequences \cite[Theorem 3.4]{MZZ}. So we may write $\boxtimes$ for both
$\boxtimes_{\Lambda}$ and $\boxtimes_{\Sigma}$.
Explicitly,
\[  (\sD\boxtimes \sE)(\bn) \cong \coprod_{j+k = n} \sD(\bj) \otimes \sE(\bk) \otimes_{\SI_j\times \SI_k}
  \SI_n.  \]

\begin{rem}
 One has the Day convolution product for covariant sequences as well.
 For $\sC, \sD \in \Lambda[\sV]$, the natural map $q \colon  \sD\boxtimes_{\SI} \sE \rtarr
 \sD\boxtimes_{\LA} \sE$ is no longer an isomorphism: For example, in $(\sC \boxtimes_{\Lambda} \sD) (\mathbf{3})$, there is one
  copy of $\sC(3) \otimes \sD(0)$ and three copies of $\sC(2) \otimes \sD(0)$. All three
  copies are glued to $\sC(3) \otimes \sD(0)$ via different maps, yielding a quotient
  space of $(\sC \boxtimes_{\Sigma} \sD) (\mathbf{3})$.
\end{rem}
\medskip

When $\otimes$ commutes with finite colimits in $\sV$, the Day convolution is
associative, and there is
\begin{equation}
\label{eq:coend2}
     \sD_1 \boxtimes_\Lambda \cdots \boxtimes_\Lambda \sD_n = \int^{i_1,\dots ,i_n} \Lambda(-,\mathbf{i_1} +\dots+\mathbf{i_n} ) \otimes \bigg( \sD_1(\mathbf{i_1}) \otimes \cdots \otimes \sD_n(\mathbf{i_n}) \bigg)
\end{equation}

\begin{defn}\label{KellyLA}  Define the Kelly product $\odot_{\LA}$ on  $\Lambda^{op}[\sV]_{\sI_0}$ by
  \[ \big ( \sD  \odot_{\LA} \sE\big )(\bn) = \int^{k} \sD(\bk)\otimes \Ek(\bn). \]
The base maps  of $\sD$ and $\sE$ induce the base map of $\sD  \odot_{\LA} \sE$ as the composite
$$\xymatrix@1{\sI_0 \ar[r]^-{\eta} & \sD = \sD \odot_{\LA} \sI_0 \ar[r]^-{\sD \odot_{\LA} \eta} & \sD \odot_{\LA} \sE \\}.$$
Observe that if $\sD(\mathbf{0})=\mI$ and $\sE(\mathbf{0})= \mI$, then these are
identity maps, so that $\odot$ restricts to a product on unital $\LA$-sequences.  

The Kelly product $\odot_{\SI}$ on $\SI^{op}[\sV]$ is defined by replacing $\LA$ by $\SI$
and ignoring the base maps in this definition.

\end{defn}

\begin{rem}
  The Kelly product is also defined when the first variable does not have a base
  map, as a functor
\begin{equation*}
   \odot_{\LA}: \Lambda^{op}[\sV] \times  \Lambda^{op}[\sV]_{\sI_0} \to  \Lambda^{op}[\sV].
\end{equation*}
\end{rem}

\section{Modules}
\label{sec:modules}
In this section, we always assume that $(\sV , \otimes , \mathcal{I})$ is a closed complete and cocomplete symmetric monoidal category.  

\subsection{Monoids and modules}
The following results established the Kelly product as a monoidal product and
identified the corresponding monoids.
Recall that
$\sI_1(\mathbf{0})= \sI_1(\mathbf{1})=\mI$ and $\sI_1(\bn)=\emptyset$ for $n\geq2$. Let
$\sI_1^{\Sigma} \in \Sigma^{op}[\sV]$ be such that
$\sI_1(\mathbf{1})=\mI$ and $\sI_1(\bn)=\emptyset$ for $n \neq 1$.
\begin{thm}\label{thm:MZZ}
  We have the following results on Kelly products:
\begin{enumerate}
\item \label{item:MZZ-1}  The Kelly product  $\odot_{\Sigma}$ is associtaive with unit
  $\sI_1^{\Sigma}$ and strong monoidal in the first variable.
  
\item \label{item:MZZ-2}  Operads in $\sV$ are the same as monoids in $(\Sigma^{op}[\sV],\odot_{\Sigma},\sI_1^{\Sigma})$.
\item \label{item:MZZ-3} The Kelly product  $\odot_{\Lambda}$ is associtaive with unit
  $\sI_1$ and strong monoidal in the first variable.
\item \label{item:MZZ-4}
  Unital operads in $\sV$ are the same as monoids in $(\Lambda^{op}_\mI[\sV],\odot_{\Lambda},\sI_1)$ and based
operads in $\sV$ are the same as monoids in $(\Lambda^{op}[\sV]_{\sI_0},
\odot_{\Lambda},\sI_1)$. 
\item \label{item:MZZ-5} If $\sV$ is Cartesian monoidal, the Kelly products
  $\odot_{\Sigma}$ and  $\odot_{\Lambda}$ are oplax monoidal in the first variable.
\end{enumerate}
\end{thm}
\begin{proof}
Properties \autoref{item:MZZ-1}-\autoref{item:MZZ-4} are proved in
\cite{Kelly0, MZZ}. For \autoref{item:MZZ-5}, note that when $\sV$ is Cartesian
monoidal, we have diagonal morphisms $\Delta: A \xrightarrow{} A \otimes A$ for $A \in \sV$. Then we have
{\small
  \begin{align*} 
    &        \sE \odot_\Lambda (\sD \boxtimes \sD') (\mathbf{n})  \\
    &=  \int^{\mathbf{m}} \sE (\mathbf{m}) \otimes (\sD \boxtimes \sD')^{\boxtimes m}(\mathbf{n}) \\
        & \cong  \int^{\mathbf{m} } \sE (\mathbf{m}) \otimes \big(\sD^{\boxtimes m} \boxtimes \sD'^{\boxtimes
          m}\big)(\mathbf{n}) \\
       & \xrightarrow{} \int^{\mathbf{m} } \sE (\mathbf{m}) \otimes \sE(\mathbf{m})\otimes
         \big(\sD^{\boxtimes m} \boxtimes \sD'^{\boxtimes m}\big)(\mathbf{n}) \\
       & \cong \int^{\mathbf{m} } \sE (\mathbf{m}) \otimes \sE(\mathbf{m})\otimes
         \bigg(\int^{\mathbf{n_1}, \mathbf{n_2}} \Lambda(\mathbf{n}, \mathbf{n}_1
         +\mathbf{n}_2) \otimes \sD^{\boxtimes m} (\mathbf{n}_1)\otimes \sD'^{\boxtimes
         m}(\mathbf{n}_2)\bigg) \\
      & \cong \int^{\mathbf{n_1}, \mathbf{n_2}, \mathbf{m}} \Lambda(\mathbf{n},
        \mathbf{n}_1 +\mathbf{n}_2) \otimes  \sE{(\mathbf{m})}\otimes \sD^{\boxtimes
        m}(\mathbf{n}_1)  \otimes  \sE{(\mathbf{m})}\otimes \sD'^{\boxtimes m}(\mathbf{n}_2) \\
      &\xrightarrow{} \int^{\mathbf{n_1}, \mathbf{n_2}, \mathbf{m_1}, \mathbf{m_2}} \Lambda(\mathbf{n}, \mathbf{n}_1 +\mathbf{n}_2) \otimes \sE{(\mathbf{m_1})}\otimes \sD^{\boxtimes m_1}(\mathbf{n}_1) \otimes \sE{(\mathbf{m_2})}\otimes \sD'^{\boxtimes m_2}(\mathbf{n}_2)  \\
   &\cong \int^{\mathbf{n_1}, \mathbf{n_2}} \Lambda(\mathbf{n}, \mathbf{n}_1 +\mathbf{n}_2) \otimes \big(\int^{\mathbf{m_1}} \sE{(\mathbf{m_1})}\otimes \sD^{\boxtimes m_1}(\mathbf{n}_1) \big) \otimes \big(\int^{\mathbf{m_2}} \sE{(\mathbf{m_2})}\otimes \sD'^{\boxtimes m_2}(\mathbf{n}_2) \big) \\
        & = (\sE \odot_\Lambda \sD ) \boxtimes (\sE \odot_\Lambda\sD') (\mathbf{n}).
  \end{align*}
}
\noindent Here, the first arrow is induced by the diagonal morphisms and the second arrow
is a comparison of colimits. For the isomorphisms, the first one uses that $\boxtimes$
is symmetric monoidal; the second one holds by definition; the third and last
one commute colimits with $\otimes$. The proof also works for $\odot_\Sigma$ replacing all $\Lambda$ by $\Sigma$.
\end{proof}

We list some properties for both contexts of $\Sigma$-sequences and $\Lambda$-sequences.
\begin{prop} \label{prop:module-dot}
If $\sE$ is a right $\sC$-module, then $\sD \odot \sE$ is a right $\sC$-module. 
 If $\sD$ is a left $\sC$-module, then $\sD \odot \sE$ is a left $\sC$-module.
\end{prop}
\begin{proof}
 For the first case, the right $\sC$-module structure is given by 
\begin{equation*}
(\sD \odot \sE) \odot \sC \cong \sD \odot (\sE \odot \sC) \to \sD \odot \sE.
\end{equation*}
The second case is similar.
\end{proof}

\begin{prop}
  If $\sD$ and $\sE$ are right $\sC$-modules, then $\sD \boxtimes \sE$ is a right
  $\sC$-module.
\end{prop}
\begin{proof}
  Using \autoref{thm:MZZ}, the right $\sC$-module structure on
  $\sD \boxtimes \sE$ is given by $(\sD \boxtimes \sE) \odot \sC \cong (\sD \odot \sC) \boxtimes (\sE \odot \sC) \to
  \sD \boxtimes \sE$.
\end{proof}

\begin{prop}
  If $\sD$ and $\sE$ are left $\sC$-modules in a Cartesion monoidal category $\sV$, then $\sD \boxtimes \sE$ is a left
  $\sC$-module.
\end{prop}
\begin{proof}
  Using \autoref{thm:MZZ}\autoref{item:MZZ-5}, the left $\sC$-module structure on
  $\sD \boxtimes \sE$ is given by $\sC \odot (\sD \boxtimes \sE) \to (\sC \odot \sD) \boxtimes (\sC \odot \sE) \to
  \sD \boxtimes \sE$.
\end{proof}

\subsection{Comparing $\Lambda$-sequences and $\Sigma$-sequences}
 Any $\Lambda$-sequence can be viewed as a $\Sigma$-sequence by forgetting
 structures. We denote by $\iota: \sI_1^{\Sigma} \to \sI_1 $ the inclusion map of
 $\Sigma$-sequences. For $\sC, \sD \in \Lambda^{op}[\sV]_{\sI_0}$, there is a canonical 
  quotient map $q: \sC \odot_{\Sigma} \sD \xrightarrow{} \sC \odot_\Lambda \sD $ and the
 two diagrams 
\begin{equation*}
  \begin{tikzcd}
   \sI_1^{\Sigma} \odot_{\Sigma} \sD \ar[d, "\iota"] \ar[r, "\cong"] & \sD \ar[dd,equal]  & \sD \odot_{\Sigma} \sI_1^{\Sigma} \ar[d, "\iota"] \ar[r, "\cong"] & \sD \ar[dd,equal]  \\
    \sI_1 \odot_{\Sigma} \sD \ar[d,"q"]& &  \sD \odot_{\Sigma} \sI_1  \ar[d,"q"]\\
    \sI_1 \odot_{\Lambda} \sD \ar[r,"\cong"] & \sD & \sD \odot_{\Lambda} \sI_1\ar[r,"\cong"] & \sD
  \end{tikzcd}
\end{equation*}
comparing the unit maps commute. Via the maps $q$ and $\iota$, the forgetful map restricts to
\begin{equation*}
  \Mon(\Lambda^{op}[\sV]_{\sI_0}) \to \Mon(\Sigma^{op}[\sV])
\end{equation*}
as well as to the categories of modules.

\begin{thm}\label{pullback} Let $\sC$ be a monoid in $(\Lambda^{op}[\sV]_{\sI_0}, \odot_{\Lambda})$, i.e. a
  based operad.  Then the forgetful map induces
  $$\Rmod_{\sC}[\LA^{op}[\sV]] \cong \Rmod_{\sC}  [\SI^{op}[\sV]].$$
  In particular,  $ \Lambda^{op}[\sV] \cong \Rmod_{\sI_1}[\Lambda^{op}[\sV]] \cong \Rmod_{\sI_1}[\Sigma^{op}[\sV]] $.
  
\end{thm}

\begin{proof}
    An object $\sD$ in $\mathrm{RMod}_{\sC}[\Sigma^{op}[\sV]]$ is a morphism
    $\sD  \odot_\Sigma \sC \xrightarrow{} \sD$ satisfying compatibility
    conditions. Explicitly,  for $n \geqslant 0$ there are morphisms
    \begin{equation}
        \coprod_{k,n_1,\cdots,n_k} \coprod_{\Sigma_n} \sD(\mathbf{k}) \otimes \sC(\mathbf{n_1})\otimes\cdots \sC(\mathbf{n_k}) \xrightarrow{} \sD(\mathbf{n}),
    \end{equation}
    which is equivalent to a series of morphisms 
    \begin{equation}
        \psi^\sigma_{k, n_1 ,\cdots , n_k} : \sD(\mathbf{k}) \otimes \sC(\mathbf{n_1})\otimes\cdots \sC(\mathbf{n_k}) \xrightarrow{} \sD(\mathbf{n})  
    \end{equation}
    where $n = n_1 + \cdots + n_k;    \sigma \in \Sigma (\mathbf{n}, \mathbf{n_1} + \cdots
    +\mathbf{n_k}) = \Sigma_n$.

    We first prove that the family of morphisms $\{ \phi^{\lambda}_{k,n_1,\cdots ,n_k} \}$
    determines a $\Lambda$-sequence structure on $\sD$. For $\sigma_i \in
    \Lambda(\mathbf{n-1},\bn)$, we let $\sD(\sigma_i) : \sD(\bn) \xrightarrow{}   \sD(\mathbf{n-1})$  be the composite
\begin{equation}\label{eq:8}
    \begin{tikzcd}
      \sD(\bn) \ar[r, "\cong"] \ar[ddd, dotted,"\sD(\sigma_i)"']
      &\sD(\mathbf{n})\otimes \mathcal{I}  \ar[d," \mathrm{id} \otimes \eta"] \\
      &\sD(\mathbf{n}) \otimes \sC(\mathbf{0}) \ar[d, "\cong"] \ar[ldd, dotted, bend right, "\sD(\sigma_i)_{0}"']\\
      & \sD(\mathbf{n})\otimes \mathcal{I} \otimes \cdots \otimes \sC(\mathbf{0}) \otimes \cdots \otimes \mI \ar[d,"\mathrm{id} \otimes
       \eta_{\sC} \otimes \cdots \otimes \mathrm{id} \otimes \cdots \otimes \eta_{\sC}  "] \\
      \sD(\mathbf{n-1})&  \sD(\mathbf{n}) \otimes \sC(\mathbf{1}) \otimes \cdots \otimes \sC (\mathbf{0}) \otimes \cdots \otimes \sC(\mathbf{1}) \ar[l,"\psi^{\sigma_i}"] 
    \end{tikzcd}
\end{equation}

    To promote the structure further to $\sD \in \mathrm{RMod}_\sC [\Lambda^{op}[\sV]]$, we need a series of morphisms 
    \begin{equation}
        \phi^\lambda_{k, n_1 ,\cdots , n_k} : \sD(\mathbf{k}) \otimes \sC(\mathbf{n_1})\otimes\cdots \sC(\mathbf{n_k}) \xrightarrow{} \sD(\mathbf{n})  
    \end{equation}
    where $n \leqslant n_1 + \cdots + n_k ; \lambda \in \Lambda(\mathbf{n},\mathbf{n_1}+ \cdots +
    \mathbf{n_k})$. These data restrict to $\psi^\sigma_{k, n_1 ,\cdots , n_k} $ on $\Sigma
    (\mathbf{n_1} + \cdots + \mathbf{n_k} , \mathbf{n_1} + \cdots + \mathbf{n_k})$ and
    they must satisfy the following additional requirements:
    
\begin{enumerate}
\item \label{condition1} for $\lambda_1 \in \Lambda(\mathbf{m}, \mathbf{n_1} + \cdots + \mathbf{n_k})$, $ \lambda_2 \in \Lambda (\mathbf{l}, \mathbf{m})$, we have
\begin{equation*}
  \begin{tikzcd}
    \sD(\mathbf{k}) \otimes \sC(\mathbf{n_1})\otimes\cdots \otimes\sC(\mathbf{n_k}) \ar[r,"\phi^{\lambda_1}"] \ar[rd , "\phi^{\lambda_1 \lambda_2}"']& \sD(\mathbf{m}) \ar[d,"\sD(\lambda_2)"] \\
      & \sD(\mathbf{l})
  \end{tikzcd}
\end{equation*}
\item \label{condition2} for $\lambda \in \Lambda(\mathbf{m}, \mathbf{n_1} + \cdots + \mathbf{n_{k-1}})$, $\sigma_i \in \Lambda(\mathbf{k}, \mathbf{k-1})$ that omits the $i$-th target, we have
\begin{equation*}\small
    \begin{tikzcd}
      \sD(\mathbf{k}) \otimes \sC(\mathbf{n_1})\otimes\cdots \otimes\mI \otimes \cdots\otimes\sC(\mathbf{n_{k-1}})
      \ar[r,"\eta"] \ar[d,"\sD(\sigma_i) \otimes \cong"']
      & \sD(\mathbf{k}) \otimes \sC(\mathbf{n_1})\otimes\cdots \otimes \sC(\mathbf{0}) \otimes \cdots \otimes\sC(\mathbf{n_{k-1}} ) \ar[d,"\phi^{\lambda}_{k,n_{1},\cdots,0,\cdots,n_{k-1}}"] \\
    \sD(\mathbf{k-1}) \otimes \sC(\mathbf{n_1})\otimes\cdots\otimes\sC(\mathbf{n_{k-1}}) \ar[r,"\phi^\lambda_{k-1,n_1,\cdots,n_{k-1}}"'] & \sD(\mathbf{m})  
    \end{tikzcd}
\end{equation*}
\end{enumerate}
where $\mathcal{I}$ is inserted in the $i$-th term.

We already have $\phi^\sigma_{k,n_1,\cdots , n_k} =\psi^\sigma_{k,n_1,\cdots,n_k}$ for $\sigma \in \Sigma_n$,
and we use diagram \autoref{condition1} to define $\phi^\lambda_{k,n_1,\cdots , n_k}$ for
general $\lambda$.  It
remains to show that these data satisfy diagram \autoref{condition2}, and because of
diagram \autoref{condition1} it suffices to show for $m = n_1 + \cdots + n_{k-1}$ and $ \lambda \in \Sigma(\mathbf{m},\mathbf{n_1}+ \cdots +
\mathbf{n_{k-1}})$. From \autoref{eq:8}, we have the dotted arrow in the diagram below

and the upper triangle commutes. The lower triangle commutes by the
associativity of $\sD$ as a module over $\sC$ in $\Sigma^{op}[\sV]$.
\begin{equation*}
    \begin{tikzcd}
      \sD(\mathbf{k}) \otimes \sC(\mathbf{n_1})\otimes\cdots \sC(\mathbf{n_{k-1}}) \ar[r,"
      \mathrm{id} \otimes\eta \otimes \mathrm{id}"]
      \ar[d,"\sD(\sigma_i) \otimes \mathrm{id}"']
      & \sD(\mathbf{k}) \otimes \sC(\mathbf{n_1})\otimes\cdots \otimes \sC(\mathbf{0}) \otimes \cdots \sC(\mathbf{n_{k-1}} ) \ar[d,"\phi^{\lambda}_{k,\cdots}=\psi^{\lambda}_{k,\cdots}"] \ar[ld,dotted,"\sD(\sigma_i)_0"]\\
    \sD(\mathbf{k-1}) \otimes \sC(\mathbf{n_1})\otimes\cdots \sC(\mathbf{n_{k-1}}) \ar[r,"\phi^\lambda_{k-1,\cdots}=\psi^{\lambda}_{k-1,\cdots}"'] &
    \sD(\mathbf{n_1}+\cdots+ \mathbf{n_{k-1}})  
    \end{tikzcd}
  \end{equation*}
  
  Alternatively, one can show that the following diagram of forgetful maps is a
  pullback square, and reduce the theorem to the case $\sC = \sI_1$, which is
  shown in \cite[5.1.8]{Fresse0}.
    \begin{equation*}
        \begin{tikzcd}
            \mathrm{RMod}_\sC[\Lambda^{op}[\sV]] \ar[r]\ar[d] &
            \mathrm{RMod}_{\sI_1}[\Lambda^{op}[\sV]] \ar[d] \\
            \mathrm{RMod}_\sC[\Sigma^{op}[\sV]] \ar[r]
            & \mathrm{RMod}_{\sI_1}[\Sigma^{op}[\sV]] .
        \end{tikzcd}
      \end{equation*}
\end{proof}

\begin{cor}
For  $\sD \in \Lambda^{op}[\sV]$, the functor $- \odot_{\Sigma} \sD: \Sigma^{op}[\sV] \to \Sigma^{op}[\sV]$ factors through
 $\Lambda^{op}[\sV]$. 
\end{cor}
\begin{proof}
  We use \autoref{pullback} to identify $\Lambda^{op}[\sV]$ with right
  $\sI_1$-modules in $\Sigma^{op}[\sV]$, then we use \autoref{prop:module-dot}.

\end{proof}

\medskip
For an operad $\sC$ in $\sV$, there is a notion of a $\sC$-algebra in
$\sV$. Left $\sC$-modules generalizes $\sC$-algebras in the following way. There
is a functor $\iota_{0}: \sV \to \Lambda^{op}[\sV]$ defined by
$\iota_0(X)(\bn) = \begin{cases}
                  X &  n=0;\\
                  \emptyset & n \geq 1.
\end{cases}$ We also write $\iota_0$ for the functor $\sV \to \Sigma^{op}[\sV]$ from
the same construction that lands in $\Sigma$-sequences.
\begin{prop}
  Let $\sV_{\mI}$ denote the category $\sV$ under its unit $\mI$.
  For a based operad $\sC$ in $\sV$, the functor $\iota_0$ restricts to
  $\mathrm{Alg}_{\sC}[\sV_\mI] \to \mathrm{LMod}_{\sC}[\Lambda^{op}[\sV]_{\sI_0}]$;
  for an operad $\sC$ in $\sV$, the functor $\iota_0$ restricts to
  $\mathrm{Alg}_{\sC}[\sV] \to  \mathrm{LMod}_{\sC}[\Sigma^{op}[\sV]]$. 
\end{prop}
\begin{proof}
  For $\Lambda$-sequences this is \cite[Lemma 6.5]{MZZ}. The same argument works for $\Sigma$-sequences.

\end{proof}

\begin{exmp} We list some examples of monoids, box products, and Kelly products.
\begin{enumerate}
\item $\sI_0$ is not a monoid as there is no unit map $\sI_1 \to \sI_0$.
\item $\sI_1^{\Sigma}$ is a monoid in $\Sigma^{op}[\sV]$.
  $$\Lmod_{\sI_1^{\Sigma}}[\Sigma^{op}[\sV]] = \Rmod_{\sI_1^{\Sigma}} [\Sigma^{op}[\sV]]= \Sigma^{op}[\sV]$$
\item $\sI_1$ is a monoid in $\Sigma^{op}[\sV]$ and $\Lambda^{op}_{\mI}[\sV]$.
  \begin{align*}
    &\Lmod_{\sI_1} [\Lambda^{op}_{\mI}[\sV]]= \Rmod_{\sI_1} [\Lambda^{op}_{\mI}[\sV]]=
    \Lambda^{op}_{\mI}[\sV], \\
    &\Rmod_{\sI_1}[\Sigma^{op}[\sV]]=\Lambda^{op}[\sV],\\
 &  \Lmod_{\sI_1}[\Sigma^{op}[\sV]] = \Sigma^{op}[\sV]_{\comm /}
  \end{align*}
\item The commutativity operad $\comm$: $\comm(\bn) = \mI$ and the structure
  maps are all isomorphic to the identity map on $\mI$.
  \begin{align*}
    (\comm \boxtimes \comm) (\bn) & = \oplus_{m=0}^n\mI[\Sigma_n/\Sigma_{m} \times \Sigma_{n-m}] \\
    &=  \mI[\text{partition of }n \text{ into 2 ordered sets}], \\
(\comm \odot_{\Sigma} \comm) (\bn) & = \mI[\text{partition of }n \text{ into unordered sets}], \\
(\comm \odot_{\Lambda} \comm) (\bn) & = \mI[\text{partition of }n \text{ into unordered nonempty sets}]
\end{align*}
\item The associativity operad $\asso$: $\asso(\bn) = \mI[\Sigma_n]$ and the
  structure maps are induced by concatinating and
  composing permutation of letters.
\begin{align*}
\asso \boxtimes \asso & = \asso \\    
  \asso \odot_{\Sigma} \asso & = \oplus_{\bN}\asso \\ 
  \asso \odot_{\Lambda} \asso & = \asso
\end{align*}
\end{enumerate}
\end{exmp}

\begin{prop} \label{prop:base-operad}
  Let $\sC$ be an operad in a Cartesian monoidal category, then $\sB = \sC \boxtimes \comm$ is also an operad.
  When $\sC$ is based, $\sB$ is also based.
\end{prop}
\begin{proof}
 In a Cartesian monoidal category, $\comm$ is the terminal object in both
 $\Sigma^{op}[\sV]$ and $\Lambda^{op}[\sV]$. The multiplication map of $\sB$ is given by ($\odot=\odot_{\Sigma}$)
\begin{align*}
\sB \odot \sB & \cong (\sC \odot (\sC \boxtimes \comm)) \boxtimes (\comm \odot (\sC \boxtimes \comm))
  &\text{by \autoref{thm:MZZ}\autoref{item:MZZ-1}}\\
  & \to (\sC \odot \sC) \boxtimes (\sC \odot \comm) \boxtimes (\comm \odot (\sC \boxtimes \comm)) & \text{by \autoref{thm:MZZ}\autoref{item:MZZ-5}}\\
  & \to (\sC \odot \sC) \boxtimes \comm & \comm \text{ is terminal}\\
  & \to  \sC \boxtimes \comm =  \sB. & \text{structure map of }\sC
\end{align*}
The unit map
  $\eta: \sI_1^{\Sigma} \cong \sI_1^{\Sigma} \boxtimes \sI_0 \to \sC \boxtimes \comm$
is the $\boxtimes$-product of the unit map $\sI_1^{\Sigma} \to \sC$ and the unique map $\sI_0 \to \comm$.
This is associative and unital.
When $\sC$ is based, just replace $\odot_{\Sigma}$ by $\odot_{\Lambda}$ and $\sI_1^{\Sigma}$ by
$\sI_1$.

\end{proof}
\begin{rem}  \label{rmk:BvsC}
More explicitly, 
\begin{align*}
  \sB(0) & = \sC(0), \\
  \sB(1) & = \sC(0) \sqcup \sC(1), \\
 \text{ and } \sB(n) & = \sqcup_{S \subset \{1,2,\cdots,n\}} \sC(|S|).
\end{align*}
Take any $S \subset \{1,\cdots,k\}$ and $T_i \subset \{1,\cdots,n_{i}\}$ for $1 \leq i \leq k$. Let
$T \subset \{1, \cdots, \Sigma_{i=1}^kn_i\}$ be the union $\cup_{i \in S}(T_i + \Sigma_{j<i, j\in S}|T_j|)$,
where $T_i + n $ is $\{x+n|x\in T_i\}$ for a number $n$. There are the multiplication maps 
\begin{equation*}
  \begin{tikzcd}
      \sC(|S|) \times \times_{i=1}^k\sC(|T_i|) \ar[rr,"\text{project to $*$}","\text{if $i \not\in S$}"']
      &  &\sC(|S|) \times \times_{i \in S}\sC(|T_i|)\ar[r,"\mu_{\sC}"]& \sC(|T|),
  \end{tikzcd}
\end{equation*}
coproducts of which give the multiplication map
\begin{equation*}
\mu_{\sB}: \sB(k) \times \times_{i=1}^k \sB(n_i) \to \sB(\Sigma_{i=1}^kn_i).
\end{equation*}
\end{rem}
\begin{rem}
  Restriction along
  $$\sC \cong \sC \boxtimes \sI_0 \to \sC \boxtimes \comm = \sB$$
  endows $\sB$-algebras the structure of $\sC$-algebras. In contrast, a $\sC$-algebra $X$ can be made into a $\sB$-algebra via the same maneuver as
  in \autoref{rmk:BvsC}. Namely, the structure maps
  $\theta_{\sB}: \sB(n)\times X^n \to  X$ are coproducts of
\begin{equation*}
  \begin{tikzcd}
    \sC(|S|) \times \times_{i=1}^n X \ar[rr,"\text{project to $*$}","\text{if $i \not\in S$}"']
      &  &\sC(|S|) \times \times_{i \in S}X \ar[r,"\theta_{\sC}"]& X
  \end{tikzcd}
\end{equation*}
for $S \subset \{1,2,\cdots,n\}$. This embeds the category of $\sC$-algebras as a fully
faithful subcategory of $\sB$-algebras.

Although $\sI_0$ is not an operad in our sense, it still makes sense to consider
$\sC =\sI_0$ as a heuristic example.
The category of $\sI_0$-algebras are based objects. It embeds into
the category of commutative algebras as the trivial ones, where the
commutative algebra structure maps are all in the form of $X^n \to * \to X$. 
\end{rem}

\begin{rem}
 There are maps of operads $\comm \to \sB \to \comm$. Here, the first map is
 $$\comm \cong \sI_0 \boxtimes \comm \to \sC \boxtimes \comm = \sB,$$
 or explicitely $* \to \sC(0) \to \sB(n)$ at each level, and the second map uses
 that $\comm$ is terminal. This embeds commutative algebras as a faithful subcategory of $\sB$-algebras.
\end{rem}
\subsection{Relation to category of operators}
Let $\fins$ denote the category with objects $\bn$ and based maps; let
$\fin \subset \fins$ denote the category with objects $\bn$ and only morphisms $\phi: \bm \to \bn$ such that
$\phi^{-1}(0) = \{0\}$ (May--Thomason \cite{MT} call these effective morphisms and
Lurie \cite{lurie2009higher} call these active morphisms.)
Denote $\Pi$ the subcategory of $\fins$ with morphisms $\phi : \mathbf{m} \xrightarrow{} \mathbf{n}$ such that each $\phi^{-1}(j)$ has at most one element for every $1 \leqslant j \leqslant n$.  Note that $\fins$ is the skeleton of the category of finite
based sets, $\fin$ is the skeleton of the category of finite sets and $\Pi \cap \fin$ is exactly $\Lambda$.

Recall that $\sV$ is a complete and cocomplete closed monoidal category in this
section. When category of operators is involved, we assume further that $\sV$ is Cartesian monoidal.
A category is considered as a $\sV$-enriched category by tensoring the
morphism sets with $\mI$.

\begin{defn}(\cite[Section 1]{Rant2})
 A $\sV$-category of operators $\sD$ is a category enriched over $\sV$ with objects the
 same as $\fins$, such that the inclusion $\Pi \xrightarrow{} \fins$ factors through as an inclusion $\Pi \xrightarrow{} \sD$ and a surjection $\sD \xrightarrow{} \fins$.
\end{defn}
\begin{defn}
 For an operad $\sC$ in a Cartesian monoidal category $\sV$, there is an associated $\sV$-enriched symmetric monoidal category $\widehat{\sC}$ with
 morphisms
 \begin{equation*}
        \widehat{\sC} (\mathbf{n},\mathbf{m}) = \coprod_{\phi \in \fins(\mathbf{n},\mathbf{m})}  \mathscr{C}(\phi^{-1}(1)) \otimes \cdots \otimes \mathscr{C}(\phi^{-1}(m)).
      \end{equation*}
 Identity maps are given by $\mathrm{id} \in \fins(\bn,\bn)$ and each unit $\mI \to \sC (\mathbf{1}) $, and the composition is induced from the
 unit $\mI \to \sC(\mathbf{1})$, projection maps $\sC(\bn) \to *$ and the operad structure maps $\gamma$.
\end{defn}

  \begin{exmp}
    We have $\Pi = \widehat{\sI_1}$, $\fins = \widehat{\comm}$. 
  \end{exmp}
  When $\sC$ is a unital operad in a Cartesian
  monoidal category $\sV$ with $\sC(\bn) \neq \varnothing$ for all $\bn$,
  $\widehat{\sC}$ is a category of operators where the map $\Pi \to \widehat{\sC}$
  is induced by $\sI_1 \to \sC$ and the map $\widehat{\sC} \to \fins$ is induced
  by $\sC \to \comm$.
\begin{defn}\cite[Section 12]{MZZ}
  Let $\sC$ be an operad in $\sV$. The permutative envelope of $\sC$ is the
  category $\overline{\sC}$ with objects $\bn$ and morphisms given by
\begin{equation*}
\overline{\sC}(\bm, \bn) = \sC^{\boxtimes n}(\bm).
\end{equation*}
\end{defn}
\noindent See \cite{MZZ} for compositions in 
$\overline{\sC}$.  Explicitly, there is
\begin{equation*}
    \overline{\sC} (\mathbf{n},\mathbf{m}) = \coprod_{\phi \in \fin(\mathbf{n},\mathbf{m})}  \mathscr{C}(\phi^{-1}(1)) \otimes \cdots \otimes \mathscr{C}(\phi^{-1}(m))
\end{equation*}
and $\overline{\sC}$ is a ($\sV$-enriched symmetric monoidal) subcategory of $\widehat{\sC}$.

\begin{exmp}
  We have $ \Lambda = \overline{\sI_1}$, $\fin = \overline{\comm}$.
\end{exmp}

\begin{prop} \label{prop:C-hat-comm}
    For any based operad $\sC$, we have isomorphisms of $\Lambda^{op}$-sequences 
    \begin{equation}
        \overline{\sC} (-, \mathbf{n}) \cong \sC^{\boxtimes n};
    \end{equation}
    \begin{equation}\label{eq:9}
        \widehat{\sC} (-,\mathbf{n}) \cong \overline{\sC} (-, \mathbf{n}) \boxtimes \comm \cong \sC^{\boxtimes n} \boxtimes \comm.
    \end{equation}
\end{prop}
\begin{proof}
  We have $\overline{\sC} (-, \mathbf{n}) = \sC^{\boxtimes n}$ from
  \cite[Lemma 12.13]{MZZ}. For $\widehat{\sC}$, we have 
  {\small
  \begin{align*}
    \overline{\sC} (\mathbf{m}, \mathbf{n}) \boxtimes \comm =
    & \int^{\mathbf{m}_1 , \mathbf{m}_2} \Sigma(\mathbf{m},\mathbf{m}_1 +
      \mathbf{m}_2) \otimes \overline{\sC} (\mathbf{m}_1, \mathbf{n}) \otimes \comm
      (\mathbf{m}_2) \\
    = & \int^{\mathbf{m}_1 , \mathbf{m}_2} \Sigma(\mathbf{m},\mathbf{m}_1 +
        \mathbf{m}_2) \otimes \coprod_{\phi \in \fin(\mathbf{m}_1,\mathbf{n})}
        \mathscr{C}(\phi^{-1}(1)) \otimes \cdots \otimes \mathscr{C}(\phi^{-1}(n)) \\
    \cong & \coprod_{\{\text{partition of }\bm=S \vee T \}} \Big( \coprod_{\phi \in
        \fin(|S|,\mathbf{n})}  \mathscr{C}(\phi^{-1}(1)) \otimes \cdots \otimes
        \mathscr{C}(\phi^{-1}(n)) \Big)\\
    \cong & \coprod_{\phi \in \fins(\mathbf{m},\mathbf{n})}  \mathscr{C}(\phi^{-1}(1)) \otimes \cdots \otimes
        \mathscr{C}(\phi^{-1}(n)) \\
    = &\widehat{\sC} (\mathbf{m},\mathbf{n}). \qedhere
  \end{align*}}
 \end{proof}
One can start with \autoref{eq:9} as the definition of $\widehat{\sC}$ and use
the same arguments as in \cite[Section 12.1]{MZZ} to show that $\widehat{\sC}$ has
the structure of a symmetric monoidal category.

\begin{rem}
Taking $\sB = \sC \boxtimes \comm$ as in \autoref{prop:base-operad}, there are maps of
$\sV$-enriched symmetric monoidal categories   $\overline{\sC} \to \overline{\sB}
\to \widehat{\sC} \to \widehat{\sB}$.
\end{rem}

\subsection{Abundance or scarcity of monads}
We review a general monad construction in \cite{Rant2} and show that applied to
the permutative envolop $\overline{\sC}$ and its subcategories $\Sigma$ and $\Lambda$,
the monad is isomorphic to the monad associated to the monoid $\sC$.
We identify the category of right $\sC$-modules completely.

\begin{defn}
    Let $\sD$ be a category enriched over $\sV$, and $\Xi$ be a subcategory of
    $\sD$. We say $\Xi$ is a discrete wide subcategory of $\sD$ if it is
    discrete, contains all objects of $\sD$, and is a subcategory of $\sD$ via
    morphisms $\coprod_{\Xi(X,Y)} \mathcal{I} \xrightarrow{} \sD(X,Y)$ for objects $X, Y \in \sD$.
  \end{defn}
  
 Let $\sD$ be a category enriched over $\sV$ and $\Xi$ be a discrete wide
 subcategory of $\sD$. Recall that for functors $X: \Xi \to \sV$ and $Y: \Xi^{op}
 \to \sV$, $Y \otimes_{\Xi} X$ is the coend \footnote{We write objects of $\Xi$ as $\bm$
   because of our applications.}
\begin{equation*}
\int^{\mathbf{m} \in \Xi} Y(\mathbf{m}) \otimes X(\mathbf{m}).
\end{equation*}

\begin{defn}(\cite[Construction 2.1]{Rant2}) \label{defn:bD}
   The associated monad $\bD_\Xi$ in $\Xi[\mathscr{V}]$ is given by
    \begin{equation*}
    \bD_\Xi X(\mathbf{n})= \mathscr{D}(-, \mathbf{n}) \otimes_\Xi X.
    \end{equation*}
\end{defn}
We may also define a monad on contravariant functors.
\begin{defn} \label{defn:bDop}
   The associated monad $\bD^{op}_\Xi$ in $\Xi^{op}[\mathscr{V}]$ is given by
    \begin{equation*}
        \bD^{op}_\Xi Y(\mathbf{n})=Y \otimes_\Xi \mathscr{D}(\mathbf{n}, -).
    \end{equation*}
\end{defn}
One can check that the definitions above indeed give monads. For example, the unit map $\eta : X \xrightarrow{} \bD_{\Xi} X  $ is given by 
\begin{equation*}
    X(\mathbf{n}) = \int^{\mathbf{m}} \Xi(\mathbf{m}, \mathbf{n}) \otimes X(\mathbf{m}) \xrightarrow{} \int^{\mathbf{m}} \sD (\mathbf{m}, \mathbf{n}) \otimes X(\mathbf{m}) = \bD_{\Xi}X(n)
\end{equation*}

We prove a categorical fact identifying the algebras, a special case of which is stated in \cite[Proposition 2.5]{Rant2} without proof.
\begin{prop} \label{prop:iso-alg}
Suppose that $\Xi$ is a discrete wide subcategory of a $\sV$-enriched category $\sD$.  There are isomorphisms of categories:
  \begin{align*}
    \mathrm{Fun}_{\sV}(\sD, \sV) &\cong   \mathrm{Alg}_{\bD_{\Xi}}[\Xi[\sV]]\\
    \mathrm{Fun}_{\sV}(\sD^{op}, \sV) &\cong  \mathrm{Alg}_{\bD^{op}_{\Xi}}[\Xi^{op}[\sV]]
  \end{align*}
\end{prop}
\begin{proof}
We prove the first case and the second case is obtained by taking $\sD$ to be
  $\sD^{op}$ and $\Xi$ to be $\Xi^{op}$ in the first case. We write $\iota: \Xi \to
  \sD$ for the inclusion of the subcategory. The monad $\bD_{\Xi}$ is indeed the
  monad of the free-forgetful adjunction of $\iota^{*}:  \mathrm{Fun}_{\sV}(\sD, \sV)
  \to     \mathrm{Fun}_{\sV}(\Xi, \sV) $, and we show that this is a monadic
  adjunction by Beck's monadicity theorem and then we are done. As $\Xi$ is a wide subcategory,
  $\iota^{*}$ reflects isomorphisms. Now we take functors $F,G \in \mathrm{Fun}_{\sV}(\sD,
  \sV)$, $H \in \mathrm{Fun}_{\sV}(\Xi, \sV)$ and a split coequalizer diagram
\begin{equation*}
  \begin{tikzcd}
    \iota^{*}F \ar[r, shift left, "\iota^{*}f"] \ar[r, shift right, "\iota^{*}g"'] & \iota^{*}G \ar[l] \ar[r, shift
    left,"e"]    & H. \ar[l, shift left, "s"]
  \end{tikzcd}
\end{equation*}
We define $K \in \mathrm{Fun}_{\sV}(\sD, \sV)$ by setting $K(\bm) =
H(\bm)$ on objects and $K = e \circ H$ on morphisms as in the commutative diagram
\[\begin{tikzcd}
	{\Xi (\mathbf{m},\mathbf{n})} \ar[d,"\iota"'] && {\underline{\mathrm{Hom}}(H(\mathbf{m}), H(\mathbf{n}))} \\
	{\mathscr{D} (\mathbf{m},\mathbf{n})} \ar[rr, "G"']&& {\underline{\mathrm{Hom}}(G(\mathbf{m}), G(\mathbf{n}))}
	\arrow["H", from=1-1, to=1-3]
	\arrow[from=1-1, to=2-1]
	\arrow["K", dashed, from=2-1, to=1-3]
	\arrow["e"', from=2-3, to=1-3]
\end{tikzcd}\]
Then $K$ is a coequalizer of $f,g$ and $\iota^{*}K = H$.
\end{proof}

\medskip

Let $\sC$ be a based operad in $\sV$. We apply \autoref{defn:bDop} to the permutative envelope $\overline{\sC}$
and the category of operators $\widehat{\sC}$ and get four associated monads
$\overline{\bC}^{op}_\Sigma$, $\overline{\bC}^{op}_\Lambda$, $\widehat{\bC}^{op}_\Sigma$, $\widehat{\bC}^{op}_\Lambda$.
\autoref{prop:iso-alg} identifies algebras over these monads in the
corresponding categories.

\begin{lem} \label{lem:monad-same}
  The monads  $\overline{\bC}^{op}_\Sigma$ and $\overline{\bC}^{op}_\Lambda$ are isomorphic
  to the monads $- \odot \sC$ associated to the monoid $\sC$ as right endofunctors
  on $\Sigma^{op}[\sV]$ and $\Lambda^{op}[\sV]$.

\end{lem}
\begin{proof}
  We show for $\Lambda$ and the case for $\Sigma$ is similar. By definitions,
  we have
    \begin{equation*}
       (\overline{\bC}^{op}_\Lambda Y) (\mathbf{n})= \int^\mathbf{m} Y(\mathbf{m}) \otimes \overline{\sC}(\mathbf{n},\mathbf{m}) = \int^\mathbf{m} Y (\mathbf{m}) \otimes \sC^{\boxtimes m}(\mathbf{n}) = (Y\odot_\Lambda \sC) (\mathbf{n}).
    \end{equation*}
The structure maps $\mu : \overline{\bC}^{op}_\Lambda \overline{\bC}^{op}_\Lambda Y \xrightarrow{} \overline{\bC}^{op}_\Lambda Y$ and $\eta : \overline{C}^{op}_\Lambda Y \xrightarrow{} Y $ are exactly the right module structure maps $\mu' : Y \odot_\Lambda \sC \odot_\Lambda \sC \xrightarrow{} Y \odot_\Lambda \sC $ and $\eta' : Y \odot_\Lambda \sC \xrightarrow{} Y$.
\end{proof}

\begin{thm}\label{thm:iso}
  The following categories are isomorphic within each group.
\begin{enumerate}
\item \begin{enumerate}
    \item \label{item:iso1}$\mathrm{Fun}_{\sV}(\overline{\sC}^{op}, \sV)$.
    \item \label{item:iso2}$ \mathrm{Alg}_{\overline{\bC}^{op}_\Sigma}[\Sigma^{op}[\sV]]$
    \item \label{item:iso3} $\mathrm{Alg}_{\overline{\bC}^{op}_\Lambda}[\Lambda^{op}[\sV]]$
    \item \label{item:iso4}$\mathrm{RMod}_\sC [\Sigma^{op}[\sV]]$
    \item \label{item:iso5}$\mathrm{RMod}_\sC [\Lambda^{op}[\sV]]$
    \end{enumerate}
\item \begin{enumerate}
    \item \label{item:iso11}$\mathrm{Fun}_{\sV}(\widehat{\sC}^{op}, \sV)$.
    \item \label{item:iso22}$ \mathrm{Alg}_{\widehat{\bC}^{op}_\Sigma}[\Sigma^{op}[\sV]]$
    \item \label{item:iso33} $\mathrm{Alg}_{\widehat{\bC}^{op}_\Lambda}[\Lambda^{op}[\sV]]$
    \end{enumerate}
\end{enumerate}  
 In parituclar, this gives an alternative proof to \autoref{pullback}.
\end{thm}
\begin{proof}

\autoref{prop:iso-alg} shows the isomorphisms of
~\autoref{item:iso1}~\autoref{item:iso2}~\autoref{item:iso3} and of ~\autoref{item:iso11}~\autoref{item:iso22}~\autoref{item:iso33}.
\autoref{lem:monad-same} shows that $\autoref{item:iso3} \cong \autoref{item:iso5}$ and $\autoref{item:iso2}\cong \autoref{item:iso4}$.
\end{proof}

\begin{exmp} \label{ex-I1-comm} We have:
\begin{enumerate}
\item   The category of right modules over $\sI_1$ in $\Sigma^{op}[\sV]$ or in
   $\Lambda^{op}[\sV]$ is isomorphic to the category of contravariant functors from
   $\Lambda$ to $\sV$. 
\item The category of right modules over $\comm$ in
   $\Sigma^{op}[\sV]$ or in $\Lambda^{op}[\sV]$ is isomorphic to the category of
   contravariant functors from $\fin$ to $\sV$.
\end{enumerate}
 \end{exmp}

\section{Monoidal localization}
\label{sec:mono-local}
\subsection{Localization}
Let $\sV$ be a category and $S$ be a family of morphisms in
$\sV$. Gabriel and Zisman \cite{gabriel2012calculus}
  constructed a category $\sV [S^{-1}]$ called the localization of $\sV$ with respect
to $S$ and a faithful functor $L: \sV \xrightarrow{} \sV [S^{-1}]$ with the
following properties: 

\begin{enumerate}
    \item $L(s)$ is an isomorphism for any morphism $s$ in $S$.
    \item For any functor $F: \sV \xrightarrow{} \sW$ such that for any $s$ in
      $S$ the image $F(s)$ is an isomorphism, there exists a unique functor
      $\tilde{F} : \sV [S^{-1}] \xrightarrow{} \sW$ such that $F= \tilde{F} \circ
      L$.
\end{enumerate}
 
The localization $\sV [S^{-1}]$ can be characterized as follows. The objects of
$\sV [S^{-1}]$ are the same as those of $\sV$. For morphisms, we consider the
``necklaces'' in $\sV$, that is, diagrams in $\sV$ of the form

\begin{center}
    $\xleftarrow{} \cdots \xleftarrow{} \xrightarrow{} \cdots \xrightarrow{} \xleftarrow{} \cdots \xleftarrow{} \xrightarrow{} \cdots \xrightarrow{} $
\end{center}
where each $\xrightarrow{}$ represents a morphism in $\mathrm{Mor}(\sV)$ and each $\xleftarrow{}$ represents a morphism in $S$. For $A,B \in \mathrm{Ob}(\sV [S^{-1}]) = \mathrm{Ob}(\sV)$, we take $\mathrm{Map}_{\sV [S^{-1}]} (A,B)$ to be the set of all necklaces in $\sV$ that connect $A$ and $B$ quotient the following relationships.

\begin{enumerate}
    \item $\xrightarrow[]{f} \xrightarrow[]{g} \sim \xrightarrow[]{g\circ f}$.
    \item For $s: A \xrightarrow{} B$ in $S$, $\xrightarrow[]{s} \xleftarrow[]{s} \sim id_A$, $\xleftarrow[]{s} \xrightarrow[]{s} \sim id_B$.
\end{enumerate}
The composition is taken in the obvious manner. We take $L: \sV \xrightarrow{} \sV [S^{-1}]$ as the embedding. It can be checked that the above universal property is satisfied. 

\medskip
Our aim is to localize a symmetric monoidal category $\sV$ so that tensor products
commute with small or finite colimits. A straightforward attempt is
to simply take $S$ as the collection of all morphisms of the form $\{
\mathrm{colim}_{i \in I} (A_i \otimes X) \to \mathrm{colim}_{i \in I}(A_i) \otimes X\}$, where
$I $ runs over all small categories and $X$ runs over objects of $\sV$. To
ensure that $\sV [S^{-1}]$ have enough colimits, Gabriel and Zisman introduced
the following property of $S$.

\begin{defn}\cite[Definition 2.2.]{gabriel2012calculus} 
\label{defn:multiplicative}
  Let $\sV$ be a category and $S$ be a collection of morphisms. The collection
  $S$ is called a \emph{right multiplicative system} (or $S$ admits a calculus
  of left fractions) if $S$ satisfies the following properties.
\begin{enumerate}

\item \label{item:mult-1} The identity morphisms of $\sV$ are all in $S$.
\item \label{item:mult-2} If $f: A \xrightarrow{} B $, $ g: B\xrightarrow{} C$ are two morphisms in $S$, then $g \circ f : A \xrightarrow{} C$ is in $S$.
\item \label{item:mult-3} Suppose we are given two morphisms $f_1 : X \xrightarrow{} Y$ and $s: X \xrightarrow{} X'$ with $s$ belonging to $S$, then there are two morphisms $f_2 : X' \xrightarrow{} Y'$ and $t: Y \xrightarrow{} Y'$ where $t$ belongs to $S$ and the following diagram commutes

\begin{center}
\begin{tikzpicture}
   \node (a)at(0,0){$X'$};
   \node (b)at(2,0){$Y'$};
   \node (c)at(0,2){$X$};
   \node (d)at(2,2){$Y$};
   \draw[->][dashed](a)--(b)node[midway,below]{$f_2$};
   \draw[->](c)--(a)node[midway,left]{$s$};
   \draw[->](c)--(d)node[midway,above]{$f_1$};
   \draw[->][dashed](d)--(b)node[midway,right]{$t$};
\end{tikzpicture}
\end{center}

\item \label{item:mult-4} Suppose we are given two morphisms $g_1 , g_2 : X \rightrightarrows{}
  Y $ in $\sV$ and a morphism $s: W \xrightarrow{} X$ in $S$ such that $g_1 \circ s
  = g_2 \circ s$, then there is an object $Z$ and a morphism $t: Y \xrightarrow{} Z$ in $S$ such that $t \circ g_1 = t \circ g_2$.  

\end{enumerate}

We define \emph{left multiplicative systems} (or a collection of morphisms that
admits a calculus of right fractions) dually. We say $S$ is a \emph{multiplicative
system} if it is both a right and a left multiplicative system.
\end{defn}

\begin{defn}
  Fix a cardinal $\kappa$, a multiplicative system $S \subset \mathrm{Mor}(\sV)$ is called \emph{$\kappa$-closed}
if it is closed under $\kappa$-small products and coproducts of morphisms; it is called \emph{closed}
if it is closed under small products and coproducts of morphisms.
\end{defn}

\begin{prop} \label{prop:cosmos}
  Let $S$ be a multiplicative system in $\sV$.
\begin{itemize}
\item  The localization $L: \sV \to \sV[S^{-1}]$
  preserves finite limits and colimits that exist in $\sV$(\cite[Prop 3.1]{GZCalculus}).
\item  If $S$ is $\kappa$-closed, the localization $L: \sV \to \sV[S^{-1}]$
  preserves $\kappa$-small limits and colimits that exist in $\sV$ (\cite[Prop 3.5.1]{Krause}).
\end{itemize} 
\end{prop}

\begin{defn}
   A collection of morphisms $S \subset \mathrm{Mor}(\sV)$ is called \emph{monoidal} if it
 is closed under tensoring with identity maps $\mathrm{id}_X: X \to X$ for all
 $X \in \sV$ (\cite[page 1]{day1973note}).
 \end{defn}

\begin{prop}\label{smalllimits}
  Let $S$ be monoidal. 
  Then the localization $L: \sV \to \sV[S^{-1}]$
  is strong symmetric monoidal.
\end{prop}
\begin{proof}

This is due to \cite[Corollary 1.4]{day1973note} and we present
  an alternative proof.
We note that the localization $L: \sV \xrightarrow{} \sV
[S^{-1}]$ induces a functor $L \times L : \sV \times \sV \xrightarrow{} \sV [S^{-1}] \times
\sV [S^{-1}]$. We take $S \times S \subset \mathrm{Mor} (\sV \times \sV)$
and formally we have $(\sV \times \sV) [(S \times S)^{-1}] \cong \sV [S^{-1}] \times \sV
  [S^{-1}]$. Since we have stipulated that $S$ is closed under finite
composition and tensor with identities, we have that for $(s_1, s_2) \in
  S \times S$, $L(s_1 \otimes s_2) =
  L(\mathrm{id} \otimes s_2)  \circ L(s_1 \otimes \mathrm{id})$ is invertible in $\sV
[S^{-1}]$. Therefore, by the universal property there exists a unique functor
$\tilde{\otimes} : \sV [S^{-1}] \times \sV[S^{-1}] \xrightarrow{} \sV[S^{-1}]$ such that
$\tilde{\otimes}\circ(L \times L) = L \circ \otimes$. This endows $\sV[S^{-1}]$ with an inherited symmetric monoidal structure.
\begin{equation*}
  \begin{tikzcd}
    & \sV \times \sV \ar[ld,"L \times L"'] \ar[d] \ar[r, "\otimes"]  & \sV \ar[d,"L"] \\
 \otimes:  \sV [S^{-1}] \times\sV [S^{-1}] \ar[r,equal] & (\sV \times \sV)[(S \times S)^{-1}]
 \ar[r,dotted,"\tilde{\otimes}"'] & \sV[S^{-1}] 
  \end{tikzcd} 
\end{equation*}
\end{proof}

We also denote the induced tensor product on $\sV[S^{-1}]$ by $\otimes$. From the discussions above, we have $LA \otimes LB = L(A\otimes B)$.

\subsection{Monoidal multiplicative closure}
\begin{con}\label{con:closure}

For $S \subset \mathrm{Mor}(\sV)$ containing all identity morphisms, 
we define a collection of morphisms $S_{\alpha}$ for each ordinal $\alpha$ by transfinite induction.

We set $S_0 = S$. For an ordinal $\alpha$, suppose that $S_{\beta}$ has been constructed for all
$ \beta < \alpha$.
When $\alpha$ is a limit cardinal, we define $S_{\alpha} = \cup_{\beta<\alpha}S_{\beta}$.
When $\alpha=\beta+1$, we define $S_{\alpha}$ as follows.

\begin{enumerate}
\item\label{item:construction1} Let $T_{\alpha,0} = S_{\beta}$. Let $T_{\alpha,1} $ be the collection of all morphisms of the form $s_1 : Y
  \xrightarrow{} Y \coprod_X Z$, where $s_0: X \to Z$ is in $T_{\alpha,0}$ and $f_1: X \to Y$
  is in $\sV$.
\begin{equation*}
  \begin{tikzcd}
    X \ar[r,"f_1"] \ar[d,"s_0"'] & Y \ar[d, dotted, "s_1"] \\
    Z \ar[r,dotted] & Y \coprod_XZ
  \end{tikzcd}
\end{equation*}

Note that $T_{\alpha,1}$ contains $T_{\alpha,0}$.

\item\label{item:construction2} $T_{\alpha,2}$ consists of morphisms in $T_{\alpha,1}$ and all canonical morphisms
  $$s_2 : \mathrm{coeq}(f \circ s_1 , g \circ s_1) \xrightarrow{} \mathrm{coeq} (f, g),$$
  where $f,g : X \rightrightarrows{} Y$ are two morphisms in $\sV$ and $s_1 : W
  \xrightarrow{} X$ is in $T_{\alpha,1}$.
  
\begin{equation*}
  \begin{tikzcd}
    W \ar[r,"s_1"] & X \ar[r, shift left, "f"] \ar[r, shift right, "g"'] &
    Y \ar[r] \ar[rd] & \mathrm{coeq}(f \circ s_1 , g \circ s_1) \ar[d, dotted, "s_2"]\\
    &&& \mathrm{coeq} (f, g)
  \end{tikzcd}
\end{equation*}

\item $T_{\alpha,3} $ is the collection of all morphisms of the form $s_3 : Y\prod_X Z
  \xrightarrow{} Y $, where $s_2: Z \to X$ is in $T_{\alpha,3}$ and $f_1: Y \to X$
  is in $\sV$.
  
\begin{equation*}
  \begin{tikzcd}
    Y\prod_XZ \ar[d,dotted,"s_3"'] \ar[r, dotted] & Z \ar[d,"s_2"]\\
    Y \ar[r,"f_1"'] & X
  \end{tikzcd}
\end{equation*}

Note $T_{\alpha,3}$ contains  $T_{\alpha,2}$ .
\item $T_{\alpha,4}$ consists of morphisms in $T_{\alpha,3}$ and all canonical morphisms
  $$s_4 : \mathrm{eq}(f,g) \xrightarrow{} \mathrm{eq}(s_3 \circ f , s_3 \circ g), $$ where
  $f,g : X \rightrightarrows{} Y$ are two morphisms in $\sV$ and $s_3 : Y \xrightarrow{} Z$ is in $T_{\alpha,3}$.
\begin{equation*}
  \begin{tikzcd}
    \mathrm{eq}(f,g) \ar[d, dotted, "s_4"'] \ar[rd] \\
     \mathrm{eq}(s_3 \circ f , s_3 \circ g) \ar[r] & X \ar[r, shift left, "f"] \ar[r,
     shift right, "g"'] & Y \ar[r, "s_3"] & Z
  \end{tikzcd}
\end{equation*}

\item $T_{\alpha,5}$ consists of all finite compositions of morphisms in $T_{\alpha,4}$.

\item $T_{\alpha,6}$ consists of all $\alpha$-small coproducts and products of morphisms
  in $T_{\alpha,5}$.
  
\item $T_{\alpha,7}$ consists of all morphisms 
  of the form $\mathrm{id}_X \otimes s_6 $ or $s_6\otimes \mathrm{id}_X$, where $s_6 $ is in
  $T_{\alpha,6}$ and $X$ is an object in $\sV$.
\end{enumerate}

We set $S_{\alpha} := T_{\alpha,7}$.

\end{con}

\begin{thm}
  \label{thm:closure}
For a regular cardinal $\kappa$, the $\overline{S}:= S_{\kappa}$ in
\autoref{con:closure} is a $\kappa$-closed monoidal multiplicative
system containing $S$.
It has the universal property that if $\overline{S}'$ is a $\kappa$-closed monoidal
 multiplicative system containing $S$, then the localization $L': \sV \to
 \sV[(\overline{S}')^{-1}]$ factors through $L: \sV \to \sV[\overline{S}^{-1}]$ uniquely.
\end{thm}
We refer to $\overline{S}$ as the $\kappa$-closed monoidal multiplicative closure of $S$.
We also have $\cup_\kappa S_{\kappa}$ as the closed monoidal multiplicative closure of
$S$, and it shares a similar universal property.
\begin{proof}
  By
construction, we have $S = S_{0} \subset S_{1} \subset \dots \subset S_{\kappa} = \bigcup_{\beta < \kappa}
S_{\beta}=\overline{S}$. 

(1)  We first show that $\overline{S}$ is both a right and a left
multiplicative system, i.e. $\overline{S}$ satisfies
\autoref{defn:multiplicative} items~\autoref{item:mult-1} to ~\autoref{item:mult-4}
and the dual statements.  

It is obvious from the construction that $\overline{S}$ contains all identities and is
  closed under finite composition. For any morphism $f_1 : X \xrightarrow{} Y$
  in $\sV$ and $s: X \xrightarrow{} X'$ in $S_{\beta} \subset \overline{S}$, the induced
  map $s' : Y \xrightarrow{} Y\coprod_X X'$ is in $S_{ \beta +1} \subset S_{ \kappa}$. Suppose that we are
  given two morphisms $g_1 , g_2 : X \rightrightarrows{} Y $ in $\sV$ and a morphism $s: W
  \xrightarrow{} X$ in $S_{\beta} \subset \overline{S}$ such that $g_1 \circ s = g_2 \circ s$. We
  may take $Z= \mathrm{coeq}(f,g)$ and $t: Y \cong \mathrm{coeq} (f\circ s , g \circ s)
  \xrightarrow{} \mathrm{coeq} (f,g) = Z$. Then $t$ belongs to $S_{ \beta +1} \subset
  \overline{S}$. Dually, the corresponding assertions for left multiplicative
  systems also hold.

(2) We then check $\overline{S}$ is closed under $\kappa$-small product and coproduct.
 For any ordinal  $\alpha < \kappa$ and $s_1 , \dots s_\alpha \in \overline{S}$, for each $i
 \leq \alpha$ there exists an ordinal $f(i) < \kappa$ such that $s_i \in S_{f(i)}$. Take
 $\beta_0 = \mathrm{Sup} \{f(1), \cdots, f(\alpha) \}$. As $\kappa$ is a regular cardinal, we
 must have $\beta_0 < \kappa$. By construction we have $\coprod_i s_i \in S_{\beta_0 +1} \subset S_{\kappa}$ and $\prod_i s_i \in S_{\beta_0 +1} \subset S_{\kappa}$. 
 It is easy to see that $\overline{S}$ is closed under tensoring with identities.

 (3) It remains to show the universal property. Suppose we are given another
 $\kappa$-closed monoidal multiplicative system containing $S$, denoted by
 $\overline{S}'$.  We use induction to show that for any ordinal
 $\alpha \leq \kappa$, $L'$ takes morphisms in $S_\alpha$ to isomorphisms in $\sV [\overline
 {S}'^{-1}]$. The base case is true because $S_0 = S$. Suppose that $L'(S_{\beta})$
 are isomorphisms for all $\beta < \alpha$, and we would like to show it for
 $L'(S_{\alpha})$. If $\alpha$ is a limit ordinal, the claim is true by the definition of
 $S_{\alpha}$. If $\alpha = \beta + 1$ for some $\beta$, using the fact that $L'$ preserves
 $\kappa$-small limits, colimits and the tensor product (\autoref{prop:cosmos} and
 \autoref{smalllimits}), one can show that $L'$ sends each class in
 \autoref{con:closure} to isormorphisms. For example,
 $L'$ sends $T_{\alpha,1}$ to isomorphisms because it does for $T_{\alpha,0}$ and the
 pushout of an isomorphism is also an isomorphism.

 The assertion will then follow from the universal property of
 localization.
\end{proof}

\subsection{Closed monoidal localization} 
\begin{lem} \label{lem:cofinal}
  Let $J$ be a small diagram and $\mathcal{B}$ be a collection of morphisms in $J$. Let
  $J'=J[\mathcal{B}^{-1}]$ be the diagram obtained by adding an inverse to each $f \in \mathcal{B}$.
  Then the inclusion $\iota: J \to J'$ is cofinal \footnote{We use cofinal to mean
   final in \cite[IX.3]{mac1998categories}, as it concerns colimits.}.
\end{lem} 
\begin{proof}
  We need to show that the comma category $x \downarrow \iota$ is non-empty and connected
  for all objects $x \in J'$. Obviously, it contains the object $(x \in J,\, \mathrm{id}: x \to
  x \in J')$. Suppose we are given any object $(y \in J,\ F: x \to y \in J')$. Then
  $F$ may be written as a finite composite 
\begin{equation*}
  \begin{tikzcd}
    x = x_0 \ar[r,"F_1"] & x_1 \ar[r,"F_2"] & x_2 \ar[r, "F_3"] & \cdots \ar[r,"F_n"]& x_n=y,
  \end{tikzcd}
\end{equation*}
where each $F_i$ is either in $J$ or is the formal inverse of some morphism in
$\mathcal{B}$. Write $F_0 = \mathrm{id}_x$ and $G_i = F_i \circ F_{i-1}\circ \cdots \circ F_0$.
 We show that there is a zigzag of morphisms in
$x \downarrow \iota$ from $(x, G_0)$ to $(x, G_n)$ by providing morphisms between pairs
 $(x, G_{i-1})$ and $(x, G_{i})$.
\begin{itemize}
\item 
Case 1: $F_i = g$ for $g \in J$. The diagram
  $
  \begin{tikzcd}
    x \ar[d,"G_{i-1}"'] \ar[r, equal] & x \ar[d,"G_i"] \\
    x_{i-1} \ar[r,"F_i"'] & x_i
  \end{tikzcd}
  $
  is a morphism from $(x, G_{i-1})$ to $(x, G_{i})$.
\item Case 2: $F_i = f^{-1}$ for $f \in \mathcal{B}$. The diagram
  $
  \begin{tikzcd}
    x \ar[d,"G_{i-1}"'] \ar[r, equal] & x \ar[d,"G_i"]\\
    x_{i-1}  & x_i \ar[l,"f"]
  \end{tikzcd}
  $
  is a morphism from $(x, G_{i})$ to $(x, G_{i-1})$.
  \qedhere
\end{itemize}
\end{proof}

\begin{thm}\label{commutewithtensor}
 Let $S = \{ \mathrm{colim}_{i \in I} (A_i \otimes X) \to \mathrm{colim}_{i \in I}(A_i) \otimes
 X\}$ where $I$ runs over all $\kappa$-small categories and $X$ runs over
 all objects of $\sV$. Then
\begin{enumerate}
\item \label{item:tensor-1} The localization  $L: \sV \to
  L\sV:=\sV[\overline{S}^{-1}]$ preserves $\kappa$-small limits, colimits and the
  tensor product.
\item \label{item:tensor-2} For any $\kappa$-small diagram $(A_j)_{j \in J}$ in $\sV$ and object $X \in \sV$, the induced map
  $$\mathrm{colim}_{j \in J} (LA_j \otimes LX) \to \mathrm{colim}_{j \in J}(LA_j) \otimes
 LX$$ is an isomorphism in $L\sV$.
\item \label{item:tensor-3} For any $\kappa$-small diagram $(B_i)_{i \in I}$ in $L\sV$ and object $X \in L\sV$, the natural map
  $$\mathrm{colim}_{i \in I} (B_i \otimes X) \to (\mathrm{colim}_{i \in I}B_i) \otimes
  X$$ is an isomorphism in $L\sV$.
\end{enumerate}
\end{thm}

\begin{proof}
    Item \autoref{item:tensor-1} follows from \autoref{prop:cosmos} and
    \autoref{smalllimits}. Item \autoref{item:tensor-2} follows from
    \autoref{item:tensor-1}.
    
    It remains to show item \autoref{item:tensor-3}. We use $X$ to denote its lift in $\sV$.
    We may use necklaces to
    lift the diagram $D_0 = (B_i)_{i \in I}$ in $L\sV$ to a $\kappa$-small diagram
    $D = (A_j)_{j\in J}$ in $\sV$, where $\mathrm{ob}(I) \subset \mathrm{ob}(J)$ and
    $A_i = B_i$ for $i \in I$.
\begin{figure}[h!]
      \centering
      \begin{tikzpicture}
        \node (a) at (0,0) {r};
        \node (b) at (3,0) {t};
        \node (c) at (1.5,0) {j};
        \draw[->] (c) --  node[above]{$s$} (a);
        \draw[->] (c) --  node[above]{$f$} (b);
        \node at (1.5,-0.5) {J};
        \node (r) at (0,-1.5) {$B_r$};
        \node (t) at (3,-1.5) {$B_t$};
        \node (j) at (1.5,-1.5) {$A_j$};
        \draw[->] (j) -- node[above]{$D(s)$}  (r);
        \draw[->] (j) -- node[above]{$D(f)$} (t);
        \node at (1.5,-2) {$D = (A_j)$ in $\sV$};
        \node at (7,0) {};

      \end{tikzpicture}

      \begin{tikzpicture}

        \node (a) at (0,0) {r};
        \node (b) at (3,0) {t};
        \draw[->] (a) -- (b);
        \node at (1.5,-0.5) {I};
        \node (r) at (0,-1.5) {$B_r$};
        \node (t) at (3,-1.5) {$B_t$};
        \draw[->] (r) -- (t);
        \node at (1.5,-2) {$D_0=(B_i)$ in $L\sV$};        
      \end{tikzpicture}
      \begin{tikzpicture}
        \node (a) at (0,0) {r};
        \node (b) at (3,0) {t};
        \node (c) at (1.5,0) {j};
        \draw[->] (a) --  node[above]{$s$} (c);
        \draw[->] (c) -- node[above]{$f$}  (b);
        \node at (1.5,-0.5) {K};
        \node (r) at (0,-1.5) {$B_r$};
        \node (t) at (3,-1.5) {$B_t$};
        \node (j) at (1.5,-1.5) {$A_j$};
        \draw[->] (r) -- node[above]{$D'(s)$} node[below]{$\sim$} (j);
        \draw[->] (j) -- node[above]{$LD(f)$} (t);
        \node at (1.5,-2) {$D'$ in $L\sV$};  
      \end{tikzpicture}
    \end{figure}
    
\noindent By construction, if we reversed all
    backward arrows in the necklaces, we obtained another $\kappa$-small category $K$
    with the same objects as $J$.
    As the backward arrows in $J$ are sent to
    isomorphisms in $L\sV$, we can construct a diagram $D' : K \to L\sV$ such
    that $D' = LD$ on objects and forward arrows in the necklaces, $D'(s) = D(s)^{-1}$
    if $s$ is a backward arrow in the necklaces $J$, and extending to other morphisms by
    composition.
    Then $I$ is a subcategory of $K$ and $(B_i)_{i \in I} = (D')|_I$.
    It may be shown that $I \subset K$ is a cofinal subcategory, and so
\begin{align}
\label{eq:4}  \mathrm{colim}_K(D') & \cong \mathrm{colim}_I(B_i) \\
\label{eq:5}  \mathrm{colim}_K(D' \otimes X) & \cong \mathrm{colim}_I(B_i \otimes X)
\end{align}

Denote the class of all the backward arrows in $J$ by $\mathcal{B}_1$ and the corresponding
reversed arrows in $K$ by $\mathcal{B}_2$. Then $J[ \mathcal{B}_1^{-1}] \cong K[ \mathcal{B}_2^{-1}]$, and the
dotted arrows exist and are equal by universal properties:
\begin{equation*}
  \begin{tikzcd}
    J \ar[r] \ar[rrd,"{LA_j}"'] & J[ \mathcal{B}_1^{-1}]  \ar[rd, dotted, "\overline{{LA_j}}"] \ar[rr, equal] &
    & K[ \mathcal{B}_2^{-1}] \ar[ld, dotted, "\overline{D'}"']
    & K \ar[l] \ar[lld,"{D'}"] \\
    & & L \sV
  \end{tikzcd}
\end{equation*}
There is 
\begin{equation}
  \label{eq:1}
\mathrm{colim}_J(LA_j) \cong
\mathrm{colim}(\overline{{LA_j}}) \cong \mathrm{colim}(\overline{D'})  \cong \mathrm{colim}_K(D'),
\end{equation}
where the first and last isomorphisms are by \autoref{lem:cofinal}.
Similarly, there is
\begin{equation}
  \label{eq:2}
\mathrm{colim}_J(LA_j \otimes X) \cong \mathrm{colim}_K(D' \otimes X).
\end{equation}
Note that in this case we used that if $s$ is a backward arrow in the necklaces,
then both $LD(s) \otimes \mathrm{id}_X$ and $D'(x) \otimes \mathrm{id}_X$ are isomorphisms
in $L\sV$, which is guaranteed by the monoidal property of $\overline{S}$.

    Combing \autoref{eq:1}, \autoref{eq:2} and item \autoref{item:tensor-2}, we have 
\begin{equation*}
\mathrm{colim}_K(D' \otimes X) \cong (\mathrm{colim}_KD') \otimes X.
\end{equation*}
With \autoref{eq:4} and \autoref{eq:5}, this finishes the proof.
\end{proof}

\section{Normal oplax monoidal structure}
\label{sec:normal-op-lax}
Let $(\sV , \otimes , \mathcal{I})$ be a symmetric monoidal category and $L$ be the monoidal
localization functor defined in \autoref{commutewithtensor}. In this section we
discuss the monoidal products of $\Lambda$-sequences when $\otimes$ in $\sV$ is not assumed
to commute with colimits.

\subsection{$n$-fold Day convolution}
For a finite set $S$ and an object $X$ in $\sV$, define $S \otimes X = \amalg_S X$ to be
the $S$-fold coproduct of the object $X$. There is
\begin{equation*}
S \otimes (T \otimes X) \cong (S \times T) \otimes X \text{ for finite sets } S,T \text{ and } X \in \sV.
\end{equation*}
\begin{prop}\label{prop:tensor-colim}
  The functor $- \otimes X: \mathrm{FinSet} \to \sV$ commutes with finite colimits; and
the functor $T \otimes -: \sV \to \sV$ commutes with colimits.
\end{prop}
\begin{proof}
 It follows from the definition.
\end{proof}
\begin{warn}
  For $X, Y \in \sV$, $S \otimes (X \otimes Y)$ can be identified with $(S \otimes X) \otimes Y$
only when $\otimes$ commutes with finite colimits in $\sV$. We shall be careful where
to put the parenthesis.
\end{warn}

\begin{defn}
    For $\sD_1 , \dots ,\sD_n \in \Lambda^{op}[\sV]$, the $n$-fold box product is defined to be 
    \begin{equation*}
        \sD_1 \boxtimes_\Lambda \cdots \boxtimes_\Lambda \sD_n = \int^{i_1,\dots ,i_n} \Lambda(-,\mathbf{i_1} +\dots+\mathbf{i_n} ) \otimes \bigg( \sD_1(\mathbf{i_1}) \otimes \cdots \otimes \sD_n(\mathbf{i_n}) \bigg)
    \end{equation*}
\end{defn}

\begin{rem}
 There are maps $\sD_1 \boxtimes_{\Lambda} \sD_2 \boxtimes_{\Lambda} \sD_3 \to (\sD_1 \boxtimes_{\Lambda} \sD_2) \boxtimes_{\Lambda} \sD_3$ and  $\sD_1 \boxtimes_{\Lambda} \sD_2 \boxtimes_{\Lambda} \sD_3 \to \sD_1 \boxtimes_{\Lambda} (\sD_2 \boxtimes_{\Lambda} \sD_3)$ by
 the universal property. These are only isomorphisms when $\otimes$ commutes with
 colimits. In fact, $(\Lambda^{op}[\sV], \boxtimes_{\Lambda}, \sI_0)$ forms a normal op-lax
 monoidal structure as defined later in \autoref{normaloplax}, but we shall omit
 the details.
\end{rem}

\begin{prop}
 For $\sD_1 , \dots ,\sD_n \in \Lambda^{op}[\sV]$, the natural map  
\begin{equation*}
\iota: \sD_1 \boxtimes_\SI \cdots \boxtimes_\SI \sD_n \to \sD_1 \boxtimes_\Lambda \cdots \boxtimes_\Lambda \sD_n
\end{equation*}
is an isomorphism of $\Sigma$-sequences.
\end{prop}
\begin{proof}
 As left Kan extension can be computed pointwise by colimits over the arrow category, it suffices to prove that in the diagram  
  \begin{equation*}
    \begin{tikzcd}
\SI^{op} \times \cdots \times \SI^{op} \ar[d, "\vee_{\SI}"'] \ar[r, "i \times \cdots \times i"] & \LA^{op} \times \cdots \times \LA^{op} \ar[d,
"\vee_{\LA}"] \\
\SI^{op} \ar[r, " i "'] & \LA^{op}
    \end{tikzcd}
\end{equation*}
the induced map of arrow categories $i:  \vee_{\Sigma}  \downarrow \bn \to \vee_{\Lambda}  \downarrow \bn$ is cofinal for every $\bn
\in \SI$.
The exact same proof of \cite[Theorem 3.4]{MZZ} works.
\end{proof}

\subsection{$2$-fold Kelly product}

Recall that when $(\sV , \otimes , \mathcal{I})$ is closed symmetric monoidal, the Kelly product is defined to be the coend
\begin{align}
\sD_1 \odot_{\Lambda} \sD_2 := {} & \int^\bk \sD_1(\mathbf{k}) \otimes \sD_2^{\boxtimes k} \notag\\ \cong {} & \int^{\mathbf{k,n_1,\dots ,
                                                            n_k}} \Lambda (-, \mathbf{n_1 + \cdots +n_k})\otimes \sD_1(\mathbf{k})\otimes \sD_2(\mathbf{n_1}) \otimes \cdots\otimes \sD_2(\mathbf{n_k})\notag
\end{align}

These equivalent definitions involve several isomorphisms commuting $\otimes$ with colimits.
Inspired by it, for $\sD_1, \sD_2 \in \Lambda^{op}[\sV]_{\sI_0}$, we set 
\begin{align*}
  \sD_1 \otimes \sD_2^{\otimes k} : \Lambda^{op} \times (\Lambda^{op})^k & \to \sV \\
  (\bm, \mathbf{n_1}, \cdots, \mathbf{n_k}) & \mapsto \sD_1(\bm) \otimes \sD_2(\mathbf{n_1}) \otimes \cdots
                                         \otimes \sD_2(\mathbf{n_k}).
\end{align*}
Let $(\sD_1 \sm \sD_2)_\bk: \Lambda^{op} \times \Lambda^{op} \to \sV$ be the left Kan extension as in:
\begin{equation*}
  \begin{tikzcd}
    \Lambda^{op} \times (\Lambda^{op})^k \ar[rr,"{\sD_1 \otimes \sD_2^{\otimes k}}"]
    \ar[d," \mathrm{id} \times \vee"'] & & \sV\\
     \Lambda^{op} \times \Lambda^{op} \ar[urr,"(\sD_1 \sm \sD_2)_\bk"']
  \end{tikzcd}
\end{equation*}
Explicitly,
\begin{equation*}
(\sD_1 \sm \sD_2)_\bk(\bm, -) = \int^{\mathbf{n_1}, \cdots, \mathbf{n_k}}\Lambda (-,\mathbf{n_1 + \cdots +n_k}) \otimes \big(\sD_1(\bm) \otimes \sD_2(\mathbf{n_1}) \otimes \cdots
                                         \otimes \sD_2(\mathbf{n_k})\big).
\end{equation*}
The functors $(\sD_1 \sm \sD_2)_\bk(-,-)$ are also functorial for $\bk$ in
$\Lambda$. To see this, the non-trivial check is for the morphisms of the form
$\sigma_{k+1}: \bk \to \mathbf{k+1}$. We examine the following diagrams of definition for
$(\sD_1 \sm \sD_2)_\bk$ and $(\sD_1 \sm \sD_2)_{\mathbf{k+1}}$.
\begin{equation*}
  \begin{tikzcd}
    \Lambda^{op} \times (\Lambda^{op})^k  \ar[rrrr,bend left, "{\sD_1 \otimes \sD_2^{\otimes k}}", ""'{name=A}]
    \ar[rr,"{\mathrm{id} \times \mathrm{id}^k \times \mathbf{0}}"]\ar[d,"\mathrm{id} \times \vee"']
    & & \Lambda^{op} \times (\Lambda^{op})^{k+1} \ar[rr,"{\sD_1 \otimes \sD_2^{\otimes k+1}}"]
    \ar[d,"\mathrm{id} \times \vee"'] \ar[from=A, Rightarrow, "\eta"]& &  \sV\\
     \Lambda^{op} \times \Lambda^{op} \ar[rr, equal] \ar[urrrr,bend right=45, "(\sD_1 \sm
     \sD_2)_{\mathbf{k}}"']
     & &  \Lambda^{op} \times \Lambda^{op} \ar[urr,near start, "(\sD_1 \sm \sD_2)_{\mathbf{k+1}}"']
  \end{tikzcd}
\end{equation*}
The natural transformation $\eta$ is induced by $\mI \to \sD_2(\mathbf{0})$, and
we have natural transformations
\begin{equation*}
(\sD_1 \sm \sD_2)_{\mathbf{k}} \to \mathrm{LKan}_{\mathrm{id} \times \vee} \sD_1 \otimes
\sD_2^{\otimes k} \otimes \sD_2(\mathbf{0}) \to (\sD_1 \sm \sD_2)_{\mathbf{k+1}},
\end{equation*}
where the first arrow is induced by $\eta$ and the second arrow is induced by $\mathrm{id} \times \mathrm{id}^k \times \mathbf{0}$.

\begin{defn}\label{2case}
  For $\sD_1, \sD_2 \in \Lambda^{op}[\sV]_{\sI_0}$, define the $2$-fold Kelly product to be
  \begin{align*} 
    \sD_1\odot_\Lambda \sD_2& = \int ^{\bk} (\sD_1 \sm \sD_2)_{\bk}(\bk,-) \\
    & = \int^{\bk,\mathbf{n_1},\dots , \mathbf{n_k}} \Lambda (-,\mathbf{n_1 + \cdots +n_k})\otimes
      \big(\sD_1(\mathbf{k})\otimes \sD_2(\mathbf{n_1}) \otimes \cdots\otimes \sD_2(\mathbf{n_k}) \big)
\end{align*}

\end{defn}

\begin{prop}\label{Lpreserve}
  Composing with $L: \sV \to L\sV$ gives functors
  $$L: \Lambda^{op}[\sV] \to \Lambda^{op}[L\sV] \text{ and } L: \Lambda^{op}[\sV]_{\sI_0} \to \Lambda^{op}[L\sV]_{\sI_0}$$
 that preserve $\boxtimes_{\Lambda}$ and $\odot_{\Lambda}$.
\end{prop}
\begin{proof}
  This is because $L$ preserves colimits and tensor products in $\sV$.
\end{proof}

We denote $\sI_1 = \Lambda(-,\mathbf{1}) \otimes \mathcal{I}$. It is a unital $\Lambda$-sequence.
\begin{exmp}\label{2unit}
     There are two canonical morphisms $\alpha_{1,0,0} : \sD \xrightarrow{} \sI_1
     \odot_\Lambda \sD$ and  $\alpha_{1,1,0} : \sD \xrightarrow{} \sD \odot_\Lambda \sI_1$, which
     we define now.
By the Yoneda Lemma, there is 
\begin{align*}
  \sD = & \int^{\mathbf{n_1}} \Lambda(-, \mathbf{n_1}) \otimes \sD(\mathbf{n_1}) \\
   = &  \int^{\bk,\mathbf{n_1} , \cdots , \mathbf{n_k}} (\Lambda(-,\mathbf{n_1 + \cdots +n_k}) \times \Lambda(\mathbf{k}, \mathbf{1}) )\otimes \sD(\mathbf{n_1}) \otimes \cdots \otimes \sD(\mathbf{n_k})
\end{align*}
By commuting tensor products with coproducts as in
$$\Lambda(\mathbf{k}, \mathbf{1}) \otimes (\mI \otimes X) \to (\Lambda(\mathbf{k}, \mathbf{1}) \otimes \mI)
\otimes X,$$
the map $\alpha_{1,0,0}$ maps $\sD$ to
\begin{align*}
   \sI_1 \odot_\Lambda \sD= {} & \int^{\bk,\mathbf{n_1} , \cdots , \mathbf{n_k}} \Lambda(-,\mathbf{n_1 + \cdots +n_k}) \otimes \bigg( (\Lambda(\mathbf{k}, \mathbf{1}) \otimes \mathcal{I})\otimes \sD(\mathbf{n_1}) \otimes \cdots \otimes \sD(\mathbf{n_k})     \bigg).
\end{align*}
And similarly for $\alpha_{1,1,0}$.
\end{exmp}

\begin{rem}
  When $\otimes$ commutes with colimits in $\sV$, both $\alpha_{1,0,0}$ and $\alpha_{1,1,0}$
  are isomorphisms. In general, $\sI_1$ is not necessarily isomorphic to $\sI_1
  \odot_{\Lambda} \sI_1$.
\end{rem}

\subsection{3-fold Kelly product}
Now we consider the 3-fold product. Since we do not have associativity any more,
we have to make separate definitions. 

Let 
\begin{equation*}
 \sD_1 \otimes \sD_2^{\otimes k} \otimes \sD_3^{\otimes t} : \Lambda^{op} \times (\Lambda^{op})^k \times (\Lambda^{op})^t \to
 \sV
\end{equation*}
\begin{equation*}
(\bm,  \mathbf{n_1}, \cdots, \mathbf{n_k}, \mathbf{j_1}, \cdots, \mathbf{j_t})
 \mapsto \sD_1(\bm) \otimes \sD_2(\mathbf{n_1}) \otimes \cdots \otimes \sD_2(\mathbf{n_k}) \otimes
  \sD_3(\mathbf{j_1}) \otimes \cdots \otimes \sD_3(\mathbf{j_t}),
\end{equation*}
and $(\sD_1 \sm \sD_2 \sm \sD_3)_{\bk,\bt}: \Lambda^{op} \times \Lambda^{op} \times \Lambda^{op} \to \sV$ be the left Kan extension as in:
\begin{equation*}
  \begin{tikzcd}
    \Lambda^{op} \times (\Lambda^{op})^k  \times (\Lambda^{op})^t  \ar[rrr,"{\sD_1 \otimes \sD_2^{\otimes k}  \otimes \sD_3^{\otimes t}}"]
    \ar[d," \mathrm{id} \times \vee \times \vee"'] & & & \sV\\
     \Lambda^{op} \times \Lambda^{op} \times \Lambda^{op} \ar[urrr,"(\sD_1 \sm \sD_2 \sm \sD_3)_{\bk,\bt}"']
  \end{tikzcd}
\end{equation*}

Similarly as before, the functors $(\sD_1 \sm \sD_2 \sm \sD_3)_{\bk,\bt}$ are
also functorial for $\bk, \bt$ in $\Lambda$. 

\begin{defn}\label{3compose}
  For $\sD_1, \sD_2, \sD_3 \in \Lambda^{op}[\sV]_{\sI_0}$, define the 3-fold Kelly
  product to be  
\begin{align*}
  \mu_3(\sD_1, \sD_2, \sD_3) = {} & \int^{\bk,\bt} (\sD_1 \sm \sD_2 \sm
                       \sD_3)_{\bk,\bt}(\bk,\bt,-)\\
  = & \int^{\substack{\bk, \bt \\
  \mathbf{n_1},\cdots,\mathbf{n_\bk}\\\mathbf{j_1},\cdots,\mathbf{j_\bt}}}
  \big(\Lambda(-,\mathbf{j_1+\cdots + j_t}) \times \Lambda(\mathbf{t},\mathbf{n_1+\cdots+n_{k}}) \big)\\
  & \otimes \bigg( \sD_1(\mathbf{k})\otimes \sD_2(\mathbf{n_1}) \otimes\cdots \otimes \sD_2(\mathbf{n_k})\otimes
    \sD_3(\mathbf{j_1}) \otimes \cdots \otimes \sD_3(\mathbf{j_t}) \bigg)
\end{align*}
\end{defn}

\begin{lem} We have the following isomorphisms:      
\begin{itemize}
\item For any object $X \in \sV$ which may be dependent
on
$$\bk,  \mathbf{n_1},\mathbf{j_{1,1}}, \cdots, \mathbf{j_{1,n_1}},
  \mathbf{n_k}, \mathbf{j_{k,1}}, \cdots, \mathbf{j_{k,n_k}},$$ there is
\begin{equation}
\label{eq:6}
\begin{split}
 &  \int^{\substack{\mathbf{r_1} , \cdots , \mathbf{r_k}}} \big( \Lambda(-,\mathbf{r_1+\cdots
  +r_k}) \\
& \phantom{==== }\times  \Lambda(\mathbf{r_1},\mathbf{j_{1,1}+\cdots
      +j_{1,n_1}}) \times \cdots \times \Lambda(\mathbf{r_k},\mathbf{j_{k,1}+\cdots
  +j_{k,n_k}})\big)  \otimes X\\
& \cong \Lambda(-, \mathbf{j_{1,1}+\cdots
      +j_{1,n_1}} + \cdots + \mathbf{j_{k,1}+\cdots  +j_{k,n_k}}) \otimes X;
\end{split}
\end{equation}

\item 
For any object $X \in \sV$ which may be dependent on
$\bk,
\mathbf{n_1},\cdots,\mathbf{n_\bk}$, there is
\begin{equation}
\label{eq:7}
\begin{split}
 & \int^{\bt,  \mathbf{j_1},\cdots,\mathbf{j_\bt}}
  \big(\Lambda(-,\mathbf{j_1+\cdots + j_t}) \times \Lambda(\mathbf{t},\mathbf{n_1+\cdots+n_{k}}) \big)
 \otimes \big( X \otimes \sD_3(\mathbf{j_1}) \otimes \cdots \otimes \sD_3(\mathbf{j_t}) \big) \\
 & \cong 
  \int^{ \mathbf{j_1},\cdots,\mathbf{j_{n_1 + \cdots + n_k}}} \Lambda(-,\mathbf{j_1+\cdots + j_{\mathbf{n_1+\cdots+n_{k}}}}) 
 \otimes \big( X \otimes \sD_3(\mathbf{j_1}) \otimes \cdots \otimes \sD_3(\mathbf{j_{\mathbf{n_1+\cdots+n_{k}}}}) \big)
\end{split}
\end{equation}
\end{itemize}
\end{lem}
\begin{proof}
 For \autoref{eq:6}, observe that
 \begin{equation}
\begin{split}
 &  \int^{\substack{\mathbf{r_1} , \cdots , \mathbf{r_k}}} \big( \Lambda(-,\mathbf{r_1+\cdots
  +r_k}) \\
& \phantom{==== }\times  \Lambda(\mathbf{r_1},\mathbf{j_{1,1}+\cdots
      +j_{1,n_1}}) \times \cdots \times \Lambda(\mathbf{r_k},\mathbf{j_{k,1}+\cdots
  +j_{k,n_k}})\big)  \otimes X\notag\\
& \cong \int^{\substack{\mathbf{r_k}}}\cdots \int^{\substack{\mathbf{r_1}} } \big( \Lambda(-,\mathbf{r_1+\cdots
  +r_k}) \\
& \phantom{==== }\times  \Lambda(\mathbf{r_1},\mathbf{j_{1,1}+\cdots
      +j_{1,n_1}}) \times \cdots \times \Lambda(\mathbf{r_k},\mathbf{j_{k,1}+\cdots
  +j_{k,n_k}})\big)  \otimes X
\end{split}
\end{equation}

By the Yoneda lemma, we have 

\begin{equation*}
    \Lambda(-,\mathbf{j_{1,1}}+\cdots+\mathbf{j_{1,n_1}}+\mathbf{r_2}+\cdots+\mathbf{r_k})\cong\int^{\mathbf{r_1}}\Lambda(-,\mathbf{r_1}+\cdots +\mathbf{r_k})\times \Lambda(\mathbf{r_1},\mathbf{j_{1,1}}+\cdots+\mathbf{j_{1,n_1}})
\end{equation*}

The claim follows by repeating the process for $\mathbf{r_2},\cdots,
\mathbf{r_k}$ and \autoref{prop:tensor-colim}.

  For \autoref{eq:7}, let $F: \Lambda^{op} \to \sV$ be given by
  $$F(\bt) = \int^{\mathbf{j_1, \cdots, j_{t}}} \Lambda(-,
  \mathbf{j_1+ \cdots + j_t}) \otimes (X \otimes \sD_3(\mathbf{j_1}) \otimes \cdots \otimes \sD_3(\mathbf{j_t})).$$
    By the Yoneda lemma, $\int^{\bt} \Lambda(\bt, \mathbf{n_1 + \cdots + n_k}) \otimes F(\bt) \cong
    F(\mathbf{n_1 + \cdots + n_k})$. The claim follows from \autoref{prop:tensor-colim}.
\end{proof}

\begin{prop}\label{3case}
  There are two natural transformations
  $$s_1 : \mu_3(\sD_1,\sD_2,\sD_3) \xrightarrow{} (\sD_1\odot_\Lambda  \sD_2)\odot_\Lambda \sD_3
  \text{ and }  s_2 :  \mu_3(\sD_1,\sD_2,\sD_3) \xrightarrow{} \sD_1\odot_\Lambda (
  \sD_2\odot_\Lambda \sD_3).$$ When $\otimes$ commutes with colimits in $\sV$, $s_1$ and $s_2$ are isomorphisms. 
\end{prop}

\begin{proof}
  We have
  {\small
  \begin{align*}
     & (\sD_1\odot_\Lambda \sD_2) \odot_\Lambda \sD_3 \\
     &= \int^{\bt, \mathbf{j_1}, \cdots, \mathbf{j_t}} \Lambda(-,\mathbf{j_1+\cdots +j_t}) \otimes
       \Big( (\sD_1\odot_\Lambda \sD_2)(\bt)
      \otimes \sD_3(\mathbf{j_1}) \otimes \cdots \otimes \sD_3(\mathbf{j_t}) \Big) \\
    &= \int^{\substack{\bt, \mathbf{j_1}, \cdots, \mathbf{j_t}}}  \Lambda(-,\mathbf{j_1+\cdots
      +j_t}) \otimes \bigg( \Big(\int^{\substack{\bk, \mathbf{n_1} , \cdots , \mathbf{n_k}}}
      \Lambda(\mathbf{t},\mathbf{n_1+\cdots+n_{k}}) 
      \\
     & \phantom{= } \otimes \sD_1(\mathbf{k})\otimes  \sD_2(\mathbf{n_1}) \otimes \cdots  \otimes \sD_2(\mathbf{n_k}) \Big)\otimes
       \sD_3(\mathbf{j_1}\big) \otimes \cdots \otimes \sD_3(\mathbf{j_t}) \bigg).
    \end{align*}
    }
Comparing with the formula in \autoref{3compose}, the natural transformation $s_1$ is given
by the natural map from  coend of tensors to the tensor with the coend. Notice that $s_1$ is pointwise in $\overline{S}$.
{\small
  \begin{align*}
     & \sD_1\odot_\Lambda (\sD_2 \odot_\Lambda \sD_3) \\
     &= \int^{\substack{\bk, \mathbf{r_1} , \cdots , \mathbf{r_k}}} \Lambda(-,\mathbf{r_1+\cdots
       +r_k}) \otimes \big(\sD_1(\bk) \otimes (\sD_2\odot_\Lambda \sD_3)(\mathbf{r_1}) \otimes \cdots \otimes
        (\sD_2\odot_\Lambda \sD_3)(\mathbf{r_k}) \big)\\ 
    &= \int^{\substack{\bk, \mathbf{r_1} , \cdots , \mathbf{r_k}}} \Lambda(-,\mathbf{r_1+\cdots
      +r_k}) \otimes \bigg( \sD_1(\bk) \otimes \\
    &\phantom{= }  \int^{\substack{\mathbf{n_1}, \mathbf{j_{1,1}}, \cdots, \mathbf{j_{1,n_1}}}}  \Lambda(\mathbf{r_1},\mathbf{j_{1,1}+\cdots
      +j_{1,n_1}})  \otimes (\sD_2(\mathbf{n_1}) \otimes \sD_3(\mathbf{j_{1,1}}) \otimes \cdots \otimes
      \sD_3(\mathbf{j_{1,n_1}})) \otimes \cdots \\
     &\phantom{= } \otimes \int^{\substack{\mathbf{n_k}, \mathbf{j_{k,1}}, \cdots, \mathbf{j_{k,n_k}}}}  \Lambda(\mathbf{r_k},\mathbf{j_{k,1}+\cdots
      +j_{k,n_k}}) \otimes (\sD_2(\mathbf{n_k}) \otimes \sD_3(\mathbf{j_{k,1}}) \otimes \cdots \otimes
      \sD_3(\mathbf{j_{k,n_k}})) \bigg).
        \end{align*}
}
We now define the natural transformation $s_2$ as the following composite:
{\small
\begin{align*}
  \mu_3(\sD_1, \sD_2, \sD_3)
  = & \int^{\substack{\bk, \bt \\
  \mathbf{n_1},\cdots,\mathbf{n_\bk}\\\mathbf{j_1},\cdots,\mathbf{j_\bt}}}
  \big(\Lambda(-,\mathbf{j_1+\cdots + j_t}) \times \Lambda(\mathbf{t},\mathbf{n_1+\cdots+n_{k}}) \big)\\
  & \otimes \bigg( \sD_1(\mathbf{k})\otimes \sD_2(\mathbf{n_1}) \otimes\cdots \otimes \sD_2(\mathbf{n_k})\otimes
    \sD_3(\mathbf{j_1}) \otimes \cdots \otimes \sD_3(\mathbf{j_t}) \bigg) \\ 
\text{ by } \autoref{eq:7} \cong &  \int^{\substack{\bk \\
  \mathbf{n_1},\cdots,\mathbf{n_\bk}\\\mathbf{j_1},\cdots,\mathbf{j_{n_1+\cdots + n_k}}}}
  \Lambda(-,\mathbf{j_1+\cdots + j_{n_1+\cdots+n_{k}}})  \\
  & \otimes \bigg( \sD_1(\mathbf{k})\otimes \sD_2(\mathbf{n_1}) \otimes\cdots \otimes \sD_2(\mathbf{n_k})\otimes
    \sD_3(\mathbf{j_1}) \otimes \cdots \otimes \sD_3(\mathbf{j_{\mathbf{n_1+\cdots+n_{k}}}}) \bigg) \\
  \text{ by } \autoref{eq:6} \cong
    & \int^{\substack{\bk, \mathbf{r_1} , \cdots ,\mathbf{r_k} \\ \mathbf{n_1},
  \mathbf{j_{1,1}}, \cdots, \mathbf{j_{1,n_1}} \\ \cdots \\ \mathbf{n_k},
  \mathbf{j_{k,1}}, \cdots, \mathbf{j_{k,n_k}}}}
  \big(\Lambda(-,\mathbf{r_1+\cdots +r_k}) \\
  & \times  \Lambda(\mathbf{r_1},\mathbf{j_{1,1}+\cdots
      +j_{1,n_1}}) \times \cdots \times  \Lambda(\mathbf{r_k},\mathbf{j_{k,1}+\cdots
      +j_{k,n_k}}) \big) \\
&  \otimes  \bigg( \sD_1(\bk)\otimes \sD_2(\mathbf{n_1}) \otimes \cdots \otimes \sD_2(\mathbf{n_k}) \otimes \sD_3(\mathbf{j_{1,1}}) \otimes \cdots \otimes
      \sD_3(\mathbf{j_{1,n_1}})) \otimes \cdots \\
&\phantom{=} \otimes \sD_3(\mathbf{j_{k,1}}) \otimes \cdots \otimes \sD_3(\mathbf{j_{k,n_k}}) \bigg)
  \\
\longrightarrow & \int^{\substack{\bk, \mathbf{r_1} , \cdots ,\mathbf{r_k} \\ \mathbf{n_1},
  \mathbf{j_{1,1}}, \cdots, \mathbf{j_{1,n_1}} \\ \cdots \\ \mathbf{n_k},
  \mathbf{j_{k,1}}, \cdots, \mathbf{j_{k,n_k}}}} 
 \Lambda(-,\mathbf{r_1+\cdots
      +r_k}) \otimes \bigg( \sD_1(\bk) \otimes \\
    &\phantom{= }   \big(\Lambda(\mathbf{r_1},\mathbf{j_{1,1}+\cdots
      +j_{1,n_1}})  \otimes (\sD_2(\mathbf{n_1}) \otimes \sD_3(\mathbf{j_{1,1}}) \otimes \cdots \otimes
      \sD_3(\mathbf{j_{1,n_1}})) \big) \otimes \cdots \\
     &\phantom{= } \otimes \big( \Lambda(\mathbf{r_k},\mathbf{j_{k,1}+\cdots
      +j_{k,n_k}}) \otimes (\sD_2(\mathbf{n_k}) \otimes \sD_3(\mathbf{j_{k,1}}) \otimes \cdots \otimes
       \sD_3(\mathbf{j_{k,n_k}})) \big) \bigg) \\ 
\longrightarrow & \sD_1\odot_\Lambda (\sD_2 \odot_\Lambda \sD_3),  
\end{align*}
}

\noindent
where the first arrow is from coproduct of tensors to tensor with coproducts and
the second arrow is from coend of tensors to the tensor with the coend.
Notice that $s_2$ is also pointwise in $\overline{S}$.

\end{proof}

\begin{exmp}\label{3unit}
 Writing $ \mu_2 (\sD_1 , \sD_2) =\sD_1 \odot_{\Lambda} \sD_2$, there are canonical maps
  \begin{align*}
  \alpha_{2,0,0} :& \mu_2 (\sD_1 , \sD_2) \xrightarrow{} \mu_3(\sI_1,\sD_1, \sD_2),\\
  \alpha_{2,1,0} :& \mu_2 (\sD_1 , \sD_2) \xrightarrow{} \mu_3(\sD_1, \sI_1, \sD_2),\\
  \alpha_{2,2,0} :& \mu_2 (\sD_1 , \sD_2) \xrightarrow{} \mu_3(\sD_1, \sD_2, \sI_1)   
  \end{align*}
 defined analogously as in \autoref{2unit}. They are in principle
 commuting tensor products with colimits. We shall later use them.
\end{exmp}

\begin{rem}
  \label{rem:Ching}
Replacing $\Lambda$ by $\Sigma$ in \autoref{2case} and \autoref{3compose}, we recover
Ching's definitions in \cite[Definition 2.12]{ChingComposition}. These definitions can be
generalized to $n$-fold Kelly products.
\end{rem}

\subsection{Normal oplax monoidal structure}
To describe the relationship between these $n$-fold products, we need the following definition which generalized the definition for monoidal category.

\begin{defn}(\cite{ChingComposition})\label{normaloplax}
 A normal oplax monoidal category $(\sE, \mu_n , \alpha_{n,l,r},I)$ consists of the following data:

\begin{enumerate}
    \item A object $I$ in $\sV$, called the unit.
    \item A series of functors $\mu_0 , \mu_1, \mu_2 , \cdots $, where $\mu_n : \sE^n
      \xrightarrow{} \sE$. The functor $\mu_0: * \to \sE$ takes the value
      $I$ and $\mu_1: \sE \to \sE$ is the identity functor.
    \item A series of natural transformations
    $$\alpha_{n,l,r} : \mu_n (X_1 , \cdots , X_n) \xrightarrow{} \mu_{n-r+1} (X_1 ,
    \cdots , X_l , \mu_r (X_{l+1} , \cdots , X_{l+r}), X_{l+r+1} , \cdots , X_n)$$
    for triples $(n,l,r)$ satisfying $0 \leqslant l \leqslant l+r \leqslant n$ . 
\end{enumerate}
The data must satisfy the following conditions:
\begin{enumerate}[(a)]
    \item For every $n$ and $l$, $\alpha_{n,l,1}$ and $\alpha_{n,0,n}$ are identities;
    \item \label{item:check-alpha}
      For every valid values of $n, l, r, k, s$, the following two diagrams commute.
\begin{equation*}
    \begin{tikzcd}[nodes={font=\scriptsize}, column sep = -3em]
      \mu_n (X_1 , \cdots , X_n) \ar[dd,"\alpha_{n,k,s}"'] \ar[dr,"\alpha_{n,l,r}"] \\
      & \mu_{n-r+1} (X_1, \cdots ,  \mu_r (X_{l+1} , \cdots , X_{l+r}),
      \cdots , X_n) \ar[dd,"\alpha_{n-r+1,k-r+1,s}"]
      \\
    \mu_{n-s+1} (\cdots , \mu_s (X_{k+1} , \cdots , X_{k+s}), \cdots) \ar[dr, "\alpha_{n-s+1,l,r}"]\\
    & \mu_{n-r-s+2} ( \cdots , \mu_r (X_{l+1} , \cdots , X_{l+r}), \cdots , \mu_s (X_{k+1} , \cdots , X_{k+s}), \cdots)
  \end{tikzcd}
\end{equation*}

\begin{equation*}
  \begin{tikzcd}[nodes={font=\scriptsize}, column sep = -3em]
     \mu_n (X_1 , \cdots , X_n) \ar[dd,"\alpha_{n,k,s}"'] \ar[dr,"\alpha_{n,l,r}"] \\
     & \mu_{n-r+1} (X_1, \cdots , \mu_r (X_{l+1} , \cdots , X_{l+r}), \cdots, X_n)
     \ar[dd,"{\mu_{n-r+1} (X_1, \cdots , \alpha_{r,k-l,s}, \cdots, X_n)}"]\\
    \mu_{n-s+1} (\cdots , \mu_s (X_{k+1} , \cdots , X_{k+s}), \cdots) \ar[dr,"\alpha_{n-s+1,l,r-s+1}"] \\
    & \mu_{n-r+1} ( \cdots , \mu_{r-s+1} (X_{l+1} , \cdots ,\mu_s (X_{k+1} , \cdots , X_{k+s}),\cdots ,  X_{l+r}),  \cdots)
  \end{tikzcd}
\end{equation*}

\end{enumerate}
\end{defn}

A monoidal category is a normal oplax monoidal category
by taking
$$\mu_n (X_1 , X_2 , \cdots , X_n) := (((X_1 \otimes X_2) \otimes \cdots ) \otimes X_n)$$ and
$\alpha$ to be induced by the associators. Conversely, a normal oplax monoidal
structure comes from a monoidal structure if
the natural transformations $\alpha_{n,l,r}$ are all isomorphisms (See
\cite[Example 1.2]{ChingComposition}).

\begin{exmp}[Reduced normal oplax monoidal structure]\label{reduced}
Let $(\sE , \mu_n, \alpha_{n,l,r}, I)$ be a normal oplax monoidal category with $\sE$
admitting a $0$-object. We define a new normal oplax monoidal structure $(\sE,
\mu_n, \overline{\alpha}_{n,l,r}, I)$ by
\begin{equation*}
 \overline{\alpha}_{n,l,r} =   
\begin{cases}
\alpha_{n,l,r} & \text{ for } r > 0  \text{ or } r=n=0, \\
0 & \text{ for } r=0 \text{ and } n > 0.
\end{cases}
\end{equation*}
It is easy to see that the diagrams in \autoref{item:check-alpha} commute.
\end{exmp}
\begin{exmp}[Truncated reduced normal oplax monoidal structure]\label{truncated}
For a reduced normal oplax monoidal structure $(\sE,
\mu_n, \overline{\alpha}_{n,l,r}, I)$ as in \autoref{reduced} and fixed $k \geqslant 1$, we define a new normal oplax monoidal structure $(\sE, \mu_n^{\leqslant k}, \overline{\alpha}_{n,l,r}^{\leqslant k}, I)$ by
\begin{equation*}
\mu_{n}^{\leqslant k} (X_1, \cdots , X_n)=   
\begin{cases}
\mu_{n} (X_1, \cdots , X_n) & \text{for } 0 \leqslant n \leqslant k, \\
0 & \text{for } n \geqslant k+1;
\end{cases}
\end{equation*}

\begin{equation*}
\overline{\alpha}_{n,l,r}^{\leqslant k} =   
\begin{cases}
\overline{\alpha}_{n,l,r} & \text{for } 0 \leqslant n \leqslant k,  \\
0 & \text{for } n \geqslant k+1.
\end{cases}
\end{equation*}
The diagrams in \autoref{item:check-alpha} still commute and this indeed forms a
normal oplax monoidal structure, which we refer to as the $k$-th truncation of
$(\sE, \mu_n, \overline{\alpha}_{n,l,r}, I)$.
\end{exmp}

\begin{exmp}\label{lowercondition}
    From \autoref{2case}, \autoref{2unit}, \autoref{3case} and \autoref{3unit}, we have already obtained the data on $\Lambda^{op}[\sV]$
\begin{equation*}
\mu_n  \text{ for } n \leq 3,\ \alpha_{n,l,r} \text{ for }n\leq 2 \text{ or } (n=3 \text{ and }1\leq r \leq 3).
\end{equation*}
These data satisfy the above compatability conditions. In
\autoref{thm:complete-mu}, we will show that the data can be extended to a
normal oplax monoidal structure.
\end{exmp}

\begin{defn}
Let $(\sE,\mu_n , \alpha_{n,l,r}, I)$ and $(\sF, \nu_n , \beta_{n,l,r},J)$ be two
normal oplax monoidal categories. A morphism $(F,t)$ from $\sE$ to $\sF$ is a functor
$F: \sE \xrightarrow{} \sF$ with natural transformations $t_n: F\circ \mu_n \to \nu_n \circ
F^n$ such that
\begin{enumerate}
    \item $F$ sends $I$ to $J$.
    \item $t_n$ is the identity for $n \leq 1$.
    \item 
      The following diagram commutes for all $(n,l,r)$ satisfies $0 \leqslant l \leqslant l+r \leqslant n$:
\begin{equation*}
\begin{tikzcd}[nodes={font=\scriptsize}]
 F \mu_n (X_1 , \cdots , X_n) 
     \arrow[dd, "{t_n (X_1 , \cdots , X_n)}"']
     \arrow[r, "{F \alpha_{n,l,r}}"]
   & 
   F \mu_{n-r+1} (X_1 , \cdots , \mu_r (X_{l+1} , \cdots , X_{l+r}), \cdots , X_n) 
   \arrow[d, "{t_{n-r+1} (X_1, \cdots , \mu_r(X_{l+1} , \cdots , X_{l+r}), \cdots , X_n)}"] \\
   & \nu_{n-r+1} (F X_1 , \cdots  ,F \mu_r (X_{l+1} , \cdots , X_{l+r}), \cdots , F X_n)
   \ar[d,"{\nu_{n-r+1} (F X_1, \cdots , t_r(X_{l+1} , \cdots , X_{l+r}), \cdots , F X_n)}"]\\
   \nu_n (F X_1 , \cdots , F X_n) 
    \arrow[r, "\beta_{n,l,r}"] 
  &  \nu_{n-r+1} (F X_1 , \cdots  , \mu_r (F X_{l+1} , \cdots , F X_{l+r}), \cdots , F X_n). 
\end{tikzcd}
\end{equation*}
\end{enumerate}
We say $F$ is a normal oplax monoidal functor if there exist the natural transformations $t_n$ completing
it to a morphism of normal oplax monoidal categories.
\end{defn}

\begin{defn}
Let $(\sE, \mu , I)$ be a normal oplax monoidal category. A monoid in $\sE$ is
an object $M$ with morphisms $m_n : \mu_n (M, M, \cdots , M) \xrightarrow{} M$ such
that the following diagram commutes for all $(n,l,r)$ satisfying $0 \leqslant l \leqslant l+r
\leqslant n$:
\begin{equation*}
\begin{tikzcd}[nodes={font=\scriptsize}]
  \mu_n (M , \cdots , M) 
    \arrow[dr, "m_n"'] 
    \arrow[r, "\alpha_{n,l,r}"] 
  & 
  \mu_{n-r+1} (M , \cdots , \mu_r (M , \cdots , M), \cdots , M) 
    \arrow[d, "{m_{n-r+1}(\cdots, m_r, \cdots)}"] \\
  & M 
\end{tikzcd}
\end{equation*}

\end{defn}
When $\sE$ is monoidal category, the two definitions of monoids are consistent.

\begin{lem}\label{lem:Ching}
 (\cite[Prop 3.4]{ChingComposition})

 Suppose we are given an object $M$ in a normal oplax monoidal category $(\sV,
 \mu_n , \alpha_{n,l,r} , I)$ with maps
 $$m_0 : I \xrightarrow{} M \text{ and } m_2 :\mu_2(M,M) \xrightarrow{} M.$$
 They determine a monoid structure on $M$ if and
 only if they satisfy the following associative laws and identity laws:
\begin{enumerate}
\item 
\begin{equation*}
  \begin{tikzcd}[nodes={font=\scriptsize}]
  & \mu_2 (\mu_2(M,M),M) \ar[rr,"{\mu_2(m_2 , M)}"] & & \mu_2(M,M) \ar[rd,"{m_2}"] \\
  \mu_3(M,M,M) \ar[ru, "{\alpha_{3,0,2}}"] \ar[rd,"{\alpha_{3,1,2}}"'] & & & & M\\
  & \mu_2(M,\mu_2(M,M)) \ar[rr,"{\mu_2(M,m_2)}"'] & & \mu_2(M,M) \ar[ru,"{m_2}"']
\end{tikzcd}
\end{equation*}
\item 
\begin{equation*}
\begin{tikzcd}[nodes={font=\scriptsize}]
  M \arrow[r, "{\alpha_{1,0,0}}"] \ar[rd,equal] &  \mu_2(I, M)  \arrow[r, "{\mu_2(m_0, M)}"]
  & \mu_2(M, M) \arrow[ld, "m_2"]\\
  &  M
\end{tikzcd}
\quad
\begin{tikzcd}[nodes={font=\scriptsize}]
  M \arrow[r, "{\alpha_{1,1,0}}"] \ar[rd,equal] &  \mu_2(M, I)  \arrow[r, "{\mu_2(M, m_0)}"]
  & \mu_2(M, M) \arrow[ld, "m_2"]\\
  & M
\end{tikzcd} 
\end{equation*}
\qed
\end{enumerate}

\end{lem}

Motivated by this result, we make the following definition.
\begin{defn} \label{defn:equi-normal-oplax}
Let $(\sE,\mu_n , \alpha_{n,l,r}, I)$ and $(\sE, \nu_n , \beta_{n,l,r}, I)$ be  two normal
oplax monoidal structures on $\sE$ with the same unit.
We say they are \emph{equivalent} if there is a zigzag of morphisms
$(\mathrm{id},t^i): (\sE,\mu_n^i) \to (\sE, \nu_n^i)$ and
$(\mathrm{id},s^i): (\sE, \mu_n^{i+1} ) \to (\sE, \nu_n^{i})$ for $0 \leq i \leq k$ with
$\mu_n^0=\mu_n$ and $\mu_n^{k+1} = \nu_n$ such that $t_n^i$ and $s_n^i$ are all
identities for $0 \leq i \leq k$ and $n \leq 3$.
\end{defn}

\begin{cor} \label{cor:equi-monoid}
  Equivalent normal oplax monoidal categories have isomorphic categories of
  monoids.
\end{cor}
\begin{proof}
  It suffices to prove that $(\mathrm{id}, t):  (\sE,\mu_n , \alpha_{n,l,r}, I) \to
  (\sE, \nu_n , \beta_{n,l,r}, I)$ with $t_n= \mathrm{id}$ for $n \leq 3$ induces
  $\Mon(\sE, \nu_n) \cong \Mon(\sE, \mu_n)$. This is true because of \autoref{lem:Ching}.
\end{proof}

In a monoidal category, higher multiplications are completely determined by the
multiplication of two objects. For normal oplax monoidal categories this is no
longer the case. For example, the truncated reduced normal oplax monoidal
structure $(\sE, \mu_n^{\leqslant k}, \overline{\alpha}_{n,l,r}^{\leqslant k} , I)$ defined in
\autoref{truncated} are equivalent but not isomorphic for all $k \geqslant 3$.

\subsection{Lifting normal oplax monoidal structures}

In this subsection, we fix a regular cardianl $\kappa > \omega$ and use small to mean
$\kappa$-small, as all diagrams considered will be $\kappa$-small.

Recall that in \autoref{commutewithtensor} we constructed a monoidal
localization
$$L: \sV \to L\sV,$$ such that $\otimes$ commutes with colimits in $L\sV$.
Therefore, $(\Lambda^{op}[L\sV]_{\sI_0}, \odot_{\Lambda}, \sI_0)$  is a monoidal category.
We will prove:
\begin{thm}\label{mainthm}
  There exists a normal oplax monoidal structure on $\Lambda^{op}[\sV]_{\sI_0}$,
  denoted $(\Lambda^{op}[\sV]_{\sI_0}, \mu_n , \alpha_{n,l,r},I)$, such that  
\begin{enumerate}
\item The unit $I = \sI_1$.
\item $\mu_2(\sD_1, \sD_2) = \sD_1 \odot_\Lambda \sD_2$ is as in \autoref{2case}.

\item \label{item:mu3} $\mu_3 (\sD_1, \sD_2, \sD_3)$ is as in \autoref{3compose}.
\item  The natural transformations $\alpha_{n,l,r}$  for $n\leq 2$, or $n=3$ and $1 \leq
  r \leq 3$ are given as in \autoref{lowercondition}.

\item $L: \Lambda^{op}[\sV]_{\sI_0} \to \Lambda^{op}[L\sV]_{\sI_0}$ is a functor of normal oplax monoidal categories.

\end{enumerate}
 And additionally, 
 for any normal oplax monoidal structure $(\Lambda^{op}[\sV], \nu_n , \beta_{n,l,r}, I)$
 on $\Lambda^{op}[\sV]$ satisfying the above properties, there exists a morphism from
 $(\Lambda^{op}[\sV], \nu_n)$ to $(\Lambda^{op}[\sV], \mu_n)$.

\end{thm}

\begin{defn}\label{annexation}
    For $n \geq 3$, we define $W_n$ to be the poset of all ways to combine $n$
    letters into one with parentheses such that each parenthesis has length of at least $2$ and at
    most ${n-1}$. The length of a parenthesis is the number of factors, where
    parenthesized letters count as one factor.
\end{defn}

\begin{exmp}
    $W_3$ has two elements $((\e_1,\e_2),\e_3)$ and
    $(\e_1,( \e_2,\e_3))$. The following poset is part of $W_4$.
\begin{equation*}
  \begin{tikzcd}
    (((\e_1,\e_2),\e_3),\e_4) & ((\e_1,\e_2),\e_3, \e_4) \ar[l] \ar[r]
    & ((\e_1,\e_2),(\e_3, \e_4))
  \end{tikzcd}
\end{equation*}
\end{exmp}

\begin{lem}\label{annex}
    When $n \geqslant4$,  $W_n$ is connected. 
\end{lem}
\begin{proof}
 By adding parentheses, any element in $W_n$ is connected to some element of the form $(\letter_1 ,
 (\cdots,(\letter_i , \letter_{i+1}),\cdots), \letter_n)$ where there is a parenthesis containing two
 single consecutive letters. By erasing parentheses, this element is further
 connected to the element $(\letter_1, \cdots,(\letter_i , \letter_{i+1}),\cdots, \letter_n)$ where there are only two pairs of
 parenthesis. Finally, any two ``adjacent'' elements of this kind are connected
 as seen in the following diagram.
\begin{center}
    \begin{tikzpicture}
   \node (a)at(-4,2){$(\letter_1 , \cdots , (\letter_{i-1} , \letter_i) , \cdots , \letter_n)$};
   \node (b)at(-3,0){$(\letter_1 , \cdots , ((\letter_{i-1} , \letter_i) , \letter_{i+1}), \cdots , \letter_n)$};
   \node (c)at(0,1.25){$(\letter_1 , \cdots , (\letter_{i-1} , \letter_i , \letter_{i+1}), \cdots , \letter_n)$};
   \node (d)at(3,0){$(\letter_1 , \cdots , (\letter_{i-1} ,( \letter_i , \letter_{i+1})), \cdots , \letter_n)$};
   \node (e)at(4,2){$(\letter_1 , \cdots , (\letter_i , \letter_{i+1}) , \cdots , \letter_n)$};
   \draw[->](a)--(b);
   \draw[->](c)--(b);
   \draw[->] (c)--(d);
   \draw[->] (e)--(d);
    \end{tikzpicture}
  \end{center}

\end{proof}

\begin{thm} \label{thm:complete-mu}
  Let $\sE$ be a category with small limits. Suppose that the data
\begin{equation*}
\mu_n  \text{ for } n \leq 3,\ \alpha_{n,l,r} \text{ for }n\leq 2 \text{ or } (n=3 \text{ and }1\leq r \leq 3).
\end{equation*}
 have been defined and the conditions in \autoref{normaloplax} are satisfied
 whenever the terms are defined.
 Then one can define $\mu_n$ for $n \geq 4$ and $\alpha_{n,l,r}$ for the remaining cases to obtain a normal oplax
 monoidal structure on $\sE$. Moreover, there is a terminal object $(\sE, \mu_n ,
 \alpha_{n,l,r},I)$ in all such structures. Therefore, all such structures are 
 equivalent in the sense of \autoref{defn:equi-normal-oplax}.
\end{thm}
\begin{proof}
  We define $\mu_n(\e_1, \cdots , \e_n)$ and $\alpha_{n,l,r}(\e_1 , \cdots , \e_n)$ by induction
  on $n$.

  For $n\geq 4$, suppose that we have defined $\mu_m$, $\alpha_{m,l,r}$ for $m
  \leq n-1$, $1 \leq r \leq m$, and $\alpha_{m,l,0}$ for $m \leq n-2$ satisfying
  \autoref{normaloplax}.
  Then for $\vec{\e}=(\e_1, \cdots , \e_n) \in \sE^n$, there is a diagram
\begin{equation*}
D[\vec\e,\mu]: W_n \to \sE
\end{equation*}
sending each element to a combination of $(\e_1 , \cdots , \e_n)$ using the
corresponding $\mu_2 , \cdots, \mu_{n-1}$ and each arrow to the (compositions of)
suitable $\alpha_{m,l,r}$ ($m \leq n-1$, $r \ge 2$), which are all defined by inductive assumption.
Define
\begin{equation*}
\mu_n(\e_1, \cdots , \e_n) = \mathrm{lim}_{W_n} D[\vec\e,\mu],
\end{equation*}
and $\alpha_{n,l,r}$ for $2 \leq r \leq n-1$ to be the canonical map
$$\mu_n (\e_1 , \cdots , \e_n) \xrightarrow{}
D[\vec\e,\mu](\letter_1 , \cdots , \letter_l ,(\letter_{l+1}, \cdots , \letter_{l+r}), \letter_{l+r+1} , \cdots , \letter_n).$$
Define $\alpha_{n,l,1}$ and $\alpha_{n,0,n}$ to be identities.
For $\alpha_{n-1,l,0}$, consider 
\begin{equation*}
  \vec\e^{\circ}=(\e_1, \cdots, \e_{n-1}) \in \sE^{n-1} \text{ and }
  \vec\e =(\e_1, \cdots, X_l, I, X_{l+1},
  \cdots , \e_{n-1}) \in \sE^{n}.
\end{equation*}
To obtain $\alpha_{n-1,l,0}: \mu_{n-1}(\vec\e^{\circ}) \to \mu_n(\vec\e)$,
it suffices to see that there are canonical maps $\alpha_{n-1,l,0}(w)$ from
$\mu_{n-1}(\vec\e^{\circ})$ to $D[\vec\e](w)$ for each element $w \in W_n$
that are compatible for all $w$. For an element $w$, there are two cases:
\begin{itemize}
\item When there exists a parenthesis of length $n-1$ after deleting
  $\letter_{l+1}$ from $w$. As $n \geq 4$, the only possiblities are 
  $w=(\letter_1,(\letter_2, \cdots, \letter_n))$, $l=0$, or $w=((\letter_1, \cdots,
  \letter_{n-1}), \letter_n)$, $l=n-1$ or $w=(\letter_1, \cdots, (\letter_{l+\epsilon},\letter_{l+\epsilon+1}), \cdots, \letter_n)$, $\epsilon=0 \text{ or }1$.
  We let $\alpha_{n-1,l,0}(w)$ be
  $$\alpha_{1,0,0}: \mu_{n-1}(\vec\e^{\circ}) \to \mu_2(I, \mu_{n-1}(\vec\e^{\circ}))$$ in the
  first case,
  $$\alpha_{1,1,0}: \mu_{n-1}(\vec\e^{\circ}) \to \mu_2(\mu_{n-1}(\vec\e^{\circ}), I)$$
  in the second case, and
  $\mu_{n-1}(\e_1, \cdots, \alpha_{2,\epsilon,0}, \cdots, \e_{n-1}) \circ \alpha_{n-1,l+\epsilon-1,1}$ in the last case.
\item When there are no parenthesis of length $n-1$ after deleting
  $\letter_{l+1}$ from $w$. We obtain an element $w^{\circ} \in W_{n-1}$ by deleting
  $\letter_{l+1}$ and any resulting parathesis of length 1.   Let
  $\alpha_{n-1,l,0}(w)$ be
\begin{equation*}
    \mu_{n-1}(\vec\e^{\circ}) \to  D[\vec\e^{\circ}](w^{\circ}) \to D[\vec\e](w),
\end{equation*}
  where the first map is some $\alpha_{3,l',2}$ when $n=4$ and the canonical map from the limit when $n>4$, and the second map is
  induced by some $\alpha_{n',l',0}$ for $n' \leq n-2$, possibly composed with some
  $\alpha_{n'',l'',1} $ for $n'' \leq n-2$ if there have been a parathesis of length 1 deleted.
\end{itemize}
The compactiblity of the $\alpha_{n-1,l,0}(w)$ across $w$ holds by inductive hypothesis because all $\mu_{m}$
appearing in the diagrams satsifies $m \leq n-1$.

We still need to check the conditions in \autoref{normaloplax}. There are
essentially three cases in \autoref{normaloplax}\autoref{item:check-alpha}:
\begin{enumerate}
\item \label{item:1} $s,r \neq 0$, we should check for $\mu_n(\e_1,\cdots,\e_n)$.
\item \label{item:2} $s=0$, $r \neq 0$, we should check for $\mu_{n-1}(\e_1,\cdots,\e_{n-1})$.
\item \label{item:3} $s=r=0$, we should check for $\mu_{n-2}(\e_1,\cdots,\e_{n-2})$.
\end{enumerate}
The diagrams in case \autoref{item:1} commute by definition of $\mu_n$.
The diagrams in case \autoref{item:2} commute by definition of $\alpha_{n-1,l,0}$.
For case \autoref{item:3}, let
\begin{equation*}
 \vec\e^{\circ}=(\e_1, \cdots, \e_{n-2}) \text{ and }
 \vec\e =(\e_1, \cdots, \e_l,I, \cdots, \e_k, I, 
 \cdots , \e_{n-2}) \in \sE^{n}.
\end{equation*}
We need to show that the two composites $\alpha_{n-1,k+1,0}\circ \alpha_{n-2,l,0}$ and
$\alpha_{n-1,l,0} \circ \alpha_{n-2,k,0}$ give the same map $\mu_{n-2}(\vec\e^{\circ}) \to
\mu_n(\vec\e)$.
For this purpose, it suffices to check that they give the same map $\mu_{n-2}(\vec\e^{\circ}) \to
D[\vec\e,\mu](w)$ for each $w \in W_n$. We may assume that $w = (\letter_1, \cdots, (\letter_i, \cdots,
\letter_j), \cdots, \letter_n)$ has two pairs of
parantheses, as any $w$ receives a map from such an element.
Work similarly as before, we attempt to
delete the letters $\letter_{l+1}$ and $\letter_{k+2}$ whose values are $I$ in $\vec\e$ from
$w$ and there are several cases:
\begin{itemize}
\item When exactly one of $\letter_{l+1}$ and $\letter_{k+2}$ is in the inner
  paranthesis. Without loss of generality we do it for $w=((\letter_1, \cdots, \letter_{l+1}), \cdots,
  \letter_n)$. Then the following diagram commutes by naturality of $\alpha_{n-l-1,k-l,0}$:
\begin{equation*}
  \begin{tikzcd}[nodes={font=\scriptsize}, column sep = -3em]
    \mu_{n-l-1}(\mu_{l}(\e_1, \cdots, \e_l), \e_{l+1}, \cdots, \e_{n-2})
    \ar[dd,"{\alpha_{n-l-1,k-l,0}}"']  \ar[rd, "{\mu_{n-l-1}(\alpha_{l,l,0},\e_{l+1},\cdots,\e_{n-2})}"]\\
    & \mu_{n-l-1}(\mu_{l+1}(\e_1, \cdots, \e_l,I), \e_{l+1}, \cdots, \e_{n-2})
            \ar[dd,"{\alpha_{n-l-1,k-l,0}}"] \\
    \mu_{n-l}(\mu_{l}(\e_1, \cdots, \e_{l}), \e_{l+1}  \cdots, \e_k,I,\cdots \e_{n-2})
    \ar[rd, "{\mu_{n-l}(\alpha_{l,l,0}, \cdots)}"] \\
    & \mu_{n-l}(\mu_{l+1}(\e_1, \cdots, \e_{l},I), \e_{l+1}, \cdots, \e_k,I,\cdots \e_{n-2})
  \end{tikzcd}
\end{equation*}
\item When both of  $\letter_{l+1}$ and $\letter_{k+2}$ are in the inner
  paranthesis, or none of them is in the inner paranthesis.
  Without loss of generality we do it for $w=((\letter_1, \cdots, \letter_{l+1}, \cdots, \letter_{k+2}), \cdots,
  \letter_n)$. Then the following diagram commutes by inductive hypothesis: 
\begin{equation*}
  \begin{tikzcd}[nodes={font=\scriptsize}, column sep = -3em]
    \mu_{n-k-1}(\mu_k(\e_1, \cdots, \e_k), \e_{k+1}, \cdots, \e_{n-2})
    \ar[dd,"{\mu_{n-k-1}(\alpha_{k,k,0},  \cdots)}"']  \ar[rd, "{\mu_{n-k-1}(\alpha_{k,l,0},\cdots)}"]\\
    & \mu_{n-k-1}(\mu_{k+1}(\e_1, \cdots, \e_l,I, \cdots, \e_k), \e_{k+1}, \cdots, \e_{n-2})
            \ar[dd,"{\mu_{n-k-1}(\alpha_{k+1,l,0}, \cdots)}"] \\
    \mu_{n-k-1}(\mu_{k+1}(\e_1, \cdots, \e_k, I), \e_{k+1}  \cdots, \e_{n-2})
    \ar[rd, "{\mu_{n-k-1}(\alpha_{k+1,l,0},  \cdots )}"] \\
    & \mu_{n-k-1}(\mu_{k+2}(\e_1, \cdots, \e_{l},I, \e_{l+1}, \cdots, \e_k,I),\cdots \e_{n-2})
  \end{tikzcd}
\end{equation*}
\end{itemize}

We now check the universal property. Let $(\sE, \nu_n , \beta_{n,l,r}, I)$ be a normal
oplax monoidal structure with $\nu_n=\mu_n$ and $\beta_{n,l,r} = \alpha_{n,l,r}$ for $n\leq 3$.
We show there is a morphism $s_n$ from this normal oplax structure to $(\sE, \mu_n
, \alpha_{n,l,r}, I)$ . When $n \leqslant 3$, $s_n$ is the identity.    
Suppose we have constructed $s_m$ for all $m \leqslant n-1$, these $s_m$ induce natural
morphisms from $D[\vec\e,\nu]$ to $D[\vec\e,\mu]$. Therefore we have natural maps between
their limits. On the other hand, the natural transformations $\beta_{m,l,r}$ for $m
\leq n$ induce compatible maps from $\nu_n(\e_1,\cdots,\e_n)$ to each entry of the diagram
$D[\vec\e,\nu]$, thus to its limit. We set $s_n$ to be the dotted arrow. 
\begin{equation*}
  \begin{tikzcd}
    \nu_n(\e_1,\cdots,\e_n) \ar[r] \ar[d,dotted,"s_n"'] & \mathrm{lim}_{W_n} D[\vec\e,\nu] \ar[d]\\
    \mu_n(\e_1, \cdots , \e_n) \ar[r, equal]&  \mathrm{lim}_{W_n} D[\vec\e,\mu]
  \end{tikzcd}
\end{equation*}
\end{proof}

\begin{proof}[Proof of \autoref{mainthm}]
  By \autoref{thm:complete-mu}, we have a terminal normal oplax monoidal structure
   $(\Lambda^{op}[\sV]_{\sI_0}, \mu_n , \alpha_{n,l,r}, \sI_0)$ extending the datum in \autoref{lowercondition}.
  In particular, by definition there is
   $$\mu_n(\sD_1, \cdots, \sD_n) := \mathrm{lim}_{W_n}D[\vec\sD,\mu],$$
   where limit in $\Lambda^{op}[\sV]_{\sI_0}$ is taken pointwise.
We show that $L: \Lambda^{op}[\sV]_{\sI_0} \to \Lambda^{op}[L\sV]_{\sI_0}$ is a functor of normal oplax monoidal categories.

   As $(\Lambda^{op}[L\sV]_{\sI_0}, \odot_{\Lambda}, \sI_0)$ is a monoidal category, it has a
   normal oplax monoidal structure $(\upsilon_n , \beta_{n,l,r})$ and the arrows in the
   diagram $D[L\vec\sD, \upsilon]$ are all isomorphisms in $\Lambda^{op}[L\sV]$, so are the
   arrows from $\upsilon_n(L\sD_1, \cdots, L\sD_n)$ to each term of the diagram. 
   As $W_n$ is a connected poset by \autoref{annex}, we have
\begin{equation*}
\upsilon_n(L\sD_1, \cdots, L\sD_n) \cong \mathrm{lim}_{W_n} D[L\vec\sD, \upsilon].
\end{equation*}
   
By \autoref{commutewithtensor}, $L$ preserves small limits and colimits.
We have
   $$L( \mu_2(\sD_1, \sD_2)) \cong L\sD_1 \odot_\Lambda L\sD_2 \text{ and } L\mu_3 (\sD_1, \sD_2,
  \sD_3) \cong L\sD_1 \odot_\Lambda L\sD_2 \odot_\Lambda L\sD_3,$$
as well as natural maps
\begin{equation*}
L\mu_n(\sD_1, \cdots, \sD_n) \cong \upsilon_n(L\sD_1, \cdots, L\sD_n) \text{ and } L\alpha_{n,l,r} \cong \beta_{n,l,r}.
\end{equation*}
\end{proof}

\begin{cor}
  $\LA^{op}_{\mI}[\sV] \subset \Lambda^{op}[\sV]_{\sI_0}$ is a sub-normal oplax monoidal category. 
  \end{cor}
  \begin{proof}
    By inspection, the category $\LA^{op}_{\mI}[\sV]$ is closed
    under operators $\mu_n$.    
  \end{proof}

\subsection{Unital operads as monoids}
We are now ready to proof the following theorem, generalizing \cite[Theorem 0.13]{MZZ}.
\begin{thm}
\label{thm:operad} Let $\sV$ be a complete and cocomplete symmetric monoidal
category.
Let $\LA^{op}[\sV]_{\mI_0}$ and $\LA^{op}_{\mI}[\sV]$ be
equipped with the normal oplax monoidal structures defined in
\autoref{mainthm}. There are isomorphisms of categories 
\begin{align*}
  \{\text{based operads in } \sV\} & \cong \{\text{monoids in }\LA^{op}[\sV]_{\sI_0}\};\\
  \{\text{unital operads in } \sV\} & \cong \{\text{monoids in }\LA^{op}_{\mI}[\sV]\}.
\end{align*}
\end{thm}
\begin{proof}
  The monoidal localization $L: \sV \to L\sV$ preserves and reflects
  monoids. Indeed, as $L$ is faithful, if the structure maps  $Lm_2 $ and $Lm_0$
  satisfy the diagrams in \autoref{lem:Ching}, so do $m_2$ and $m_0$.
  
It is easy to see that given a monoid in $\LA^{op}[\sV]_{\mI_0}$ (or
respectively, $\LA^{op}_{\mI}[\sV]$), the structure maps $m_0 $ and $m_2$
provide the structure maps of a based (respectively a unital) operad on
$\sV$. Conversely, suppose that we are given a based (unital) operad $\sC$ in
$\sV$. As $L$ is symmetric monoidal and preserves colimits, $L\sC$ is an operad
in $L\sV$, thus a monoid in $\LA^{op}[L\sV]_{\sI_0}$ (respectively
$\LA^{op}_{\mI}[L\sV]$). This proves that $\sC$ is a monoid in $\LA^{op}[\sV]_{\sI_0}$ (respectively
$\LA^{op}_{\mI}[\sV]$).
\end{proof}

\begin{rem} \label{rem:all-same}
 As we see in \autoref{mainthm}, all normal oplax monoidal structures
 satisfying the requirements are equivalent in the sense of
 \autoref{defn:equi-normal-oplax}. Consequently,
 the category of operads is isomorphic to the category of monoids in
 $\Lambda$-sequences equipped with \emph{any} normal oplax monoidal structure by
 \autoref{cor:equi-monoid}.
\end{rem}

\begin{rem}
  Ching \cite{ChingComposition} constructed a normal oplax monoidal structure
  on $\Sigma$-sequences, denoted $(\Sigma^{op}[\sV], \mu_n,  \alpha_{n,l,r}, \sI_1^{\Sigma})$ here,
  to prove an analogue of \autoref{thm:operad} identifying operads in $\sV$ to
  monoids in $\Sigma^{op}[\sV]$. Our constructions in this section can be applied to $\Sigma$-sequences similarly
  to obtain another normal oplax monoidal structure $(\Sigma^{op}[\sV], \nu_n, \beta_{n,l,r}, \sI_1^{\Sigma})$.

Ching's structure has a morphism to our structure, since it satisfies the conditions in
 \autoref{mainthm}. We do not know whether they are isomorphic.
\end{rem}

\bibliographystyle{plain}
\bibliography{references}

@article {MZZ,
    AUTHOR = {May, J. P. and Zhang, Ruoqi and Zou, Foling},
     TITLE = {Unital operads, monoids, monads, and bar constructions},
   JOURNAL = {Adv. Math.},
  FJOURNAL = {Advances in Mathematics},
    VOLUME = {461},
      YEAR = {2025},
     PAGES = {Paper No. 110065, 55},
      ISSN = {0001-8708,1090-2082},
   MRCLASS = {18M60 (18M05 55P48)},
  MRNUMBER = {4837126},
       DOI = {10.1016/j.aim.2024.110065},
       URL = {https://doi.org/10.1016/j.aim.2024.110065},
}

@book {GZCalculus,
    AUTHOR = {Gabriel, P. and Zisman, M.},
     TITLE = {Calculus of fractions and homotopy theory},
    SERIES = {Ergebnisse der Mathematik und ihrer Grenzgebiete [Results in
              Mathematics and Related Areas]},
    VOLUME = {Band 35},
 PUBLISHER = {Springer-Verlag New York, Inc., New York},
      YEAR = {1967},
     PAGES = {x+168},
   MRCLASS = {55.40 (18.00)},
  MRNUMBER = {210125},
MRREVIEWER = {A.\ K.\ Bousfield},
}

@book {lurie2009higher,
  title={Higher topos theory},
  author={Lurie, Jacob},
  year={2009},
  publisher={Princeton University Press}
}

@incollection {Krause,
    AUTHOR = {Krause, Henning},
     TITLE = {Localization theory for triangulated categories},
 BOOKTITLE = {Triangulated categories},
    SERIES = {London Math. Soc. Lecture Note Ser.},
    VOLUME = {375},
     PAGES = {161--235},
 PUBLISHER = {Cambridge Univ. Press, Cambridge},
      YEAR = {2010},
      ISBN = {978-0-521-74431-7},
   MRCLASS = {18E35 (18E30)},
  MRNUMBER = {2681709},
MRREVIEWER = {Jue\ Le},
       DOI = {10.1017/CBO9781139107075.005},
       URL = {https://doi.org/10.1017/CBO9781139107075.005},
}

@book {mac1998categories,
  title={Categories for the working mathematician},
  author={Mac Lane, Saunders},
  volume={5},
  year={1998},
  publisher={Springer Science \& Business Media}
}

@incollection {And,
    AUTHOR = {Anderson, D. W.},
     TITLE = {Chain functors and homology theories},
 BOOKTITLE = {Symposium on {A}lgebraic {T}opology ({B}attelle {S}eattle
              {R}es. {C}enter, {S}eattle, {W}ash., 1971)},
     PAGES = {1--12. Lecture Notes in Math., Vol. 249},
 PUBLISHER = {Springer, Berlin},
      YEAR = {1971},
   MRCLASS = {55B20},
  MRNUMBER = {0339132}}

@article {BMAx,
    AUTHOR = {Berger, Clemens and Moerdijk, Ieke},
     TITLE = {Axiomatic homotopy theory for operads},
   JOURNAL = {Comment. Math. Helv.},
  FJOURNAL = {Commentarii Mathematici Helvetici},
    VOLUME = {78},
      YEAR = {2003},
    NUMBER = {4},
     PAGES = {805--831},
      ISSN = {0010-2571},
   MRCLASS = {18D50 (18G55 55P48 55U35)},
  MRNUMBER = {2016697},
MRREVIEWER = {David Chataur},
       DOI = {10.1007/s00014-003-0772-y},
       URL = {https://doi.org/10.1007/s00014-003-0772-y},
}

@article {BergMoerOp2,
    AUTHOR = {Berger, Clemens and Moerdijk, Ieke},
     TITLE = {The {B}oardman-{V}ogt resolution of operads in monoidal model
              categories},
   JOURNAL = {Topology},
  FJOURNAL = {Topology. An International Journal of Mathematics},
    VOLUME = {45},
      YEAR = {2006},
    NUMBER = {5},
     PAGES = {807--849},
      ISSN = {0040-9383},
   MRCLASS = {18D50 (18G55 55P48 55U35)},
  MRNUMBER = {2248514},
MRREVIEWER = {Agust\'{i} Roig},
       DOI = {10.1016/j.top.2006.05.001},
       URL = {https://doi.org/10.1016/j.top.2006.05.001},
}

@article {ChingComposition,
    AUTHOR = {Ching, Michael},
     TITLE = {A note on the composition product of symmetric sequences},
   JOURNAL = {J. Homotopy Relat. Struct.},
  FJOURNAL = {Journal of Homotopy and Related Structures},
    VOLUME = {7},
      YEAR = {2012},
    NUMBER = {2},
     PAGES = {237--254},
      ISSN = {2193-8407,1512-2891},
   MRCLASS = {18D50 (18D10)},
  MRNUMBER = {2988948},
MRREVIEWER = {Sara\ Madariaga},
       DOI = {10.1007/s40062-012-0007-2},
       URL = {https://doi.org/10.1007/s40062-012-0007-2},
}

@article{day1973note,
  title={Note on monoidal localisation},
  author={Day, Brian},
  journal={Bulletin of the Australian Mathematical Society},
  volume={8},
  number={1},
  pages={1--16},
  year={1973},
  publisher={Cambridge University Press}
}

@book {Fresse0,
    AUTHOR = {Fresse, Benoit},
     TITLE = {Modules over operads and functors},
    SERIES = {Lecture Notes in Mathematics},
    VOLUME = {1967},
 PUBLISHER = {Springer-Verlag, Berlin},
      YEAR = {2009},
     PAGES = {x+308},
      ISBN = {978-3-540-89055-3},
   MRCLASS = {18D50 (18G50 55P48 57T30)},
  MRNUMBER = {2494775},
MRREVIEWER = {Paul G. Goerss},
       DOI = {10.1007/978-3-540-89056-0},
       URL = {https://doi.org/10.1007/978-3-540-89056-0},
}

@book{gabriel2012calculus,
  title={Calculus of fractions and homotopy theory},
  author={Gabriel, Peter and Zisman, Michel},
  volume={35},
  year={2012},
  publisher={Springer Science \& Business Media}}

@article {Kelly0,
    AUTHOR = {Kelly, G. M.},
     TITLE = {On the operads of {J}. {P}. {M}ay},
   JOURNAL = {Repr. Theory Appl. Categ.},
  FJOURNAL = {Reprints in Theory and Applications of Categories},
    NUMBER = {13},
      YEAR = {2005},
     PAGES = {1--13},
   MRCLASS = {18D50 (18C15 18D10 18D20)},
  MRNUMBER = {2177746},
MRREVIEWER = {Agust\'{i} Roig},
}

@incollection {Kelly,
    AUTHOR = {Kelly, G. M.},
     TITLE = {Coherence theorems for lax algebras and for distributive laws},
 BOOKTITLE = {Category {S}eminar ({P}roc. {S}em., {S}ydney, 1972/1973)},
     PAGES = {281--375. Lecture Notes in Math., Vol. 420},
 PUBLISHER = {Springer, Berlin},
      YEAR = {1974},
   MRCLASS = {18D05},
  MRNUMBER = {0364394}}

@article {Kro,
    AUTHOR = {Kro, Tore August},
     TITLE = {Model structure on operads in orthogonal spectra},
   JOURNAL = {Homology Homotopy Appl.},
  FJOURNAL = {Homology, Homotopy and Applications},
    VOLUME = {9},
      YEAR = {2007},
    NUMBER = {2},
     PAGES = {397--412},
      ISSN = {1532-0073},
   MRCLASS = {18D50 (55P42 55P48 55U35)},
  MRNUMBER = {2366955},
       URL = {http://projecteuclid.org/euclid.hha/1201127343}}

@book{MayGeo,
	Address = {Berlin},
	Author = {May, J. P.},
	Mrclass = {55D35},
	Mrnumber = {0420610 (54 \#8623b)},
	Mrreviewer = {J. Stasheff},
	Note = {Lectures Notes in Mathematics, Vol. 271},
	Pages = {viii+175},
	Publisher = {Springer-Verlag},
	Title = {The geometry of iterated loop spaces},
	Year = {1972}}

@incollection{Rant2,
	Author = {May, J. P.},
	Booktitle = {New topological contexts for {G}alois theory and algebraic geometry ({BIRS} 2008)},
	Mrclass = {18C20 (18D10 18D50 19D23 55P48)},
	Mrnumber = {MR2544392},
	Pages = {283--330},
	Publisher = {Geom. Topol. Publ., Coventry},
	Series = {Geom. Topol. Monogr.},
	Title = {The construction of {$E\sb \infty$} ring spaces from bipermutative categories},
	Volume = {16},
	Year = {2009}}

@article{MT,
	Author = {May, J. P. and Thomason, R.},
	Coden = {TPLGAF},
	Doi = {10.1016/0040-9383(78)90026-5},
	Fjournal = {Topology. An International Journal of Mathematics},
	Issn = {0040-9383},
	Journal = {Topology},
	Mrclass = {55P42 (55N22 55P35)},
	Mrnumber = {508885 (80g:55015)},
	Mrreviewer = {J. F. Adams},
	Number = {3},
	Pages = {205--224},
	Title = {The uniqueness of infinite loop space machines},
	Url = {http://dx.doi.org.proxy.uchicago.edu/10.1016/0040-9383(78)90026-5},
	Volume = {17},
	Year = {1978}}

@article{PS,
Author ={Pavlov, Dmitri and Scholbach, Jakob},
Title = {Admissibility and rectification of colored symmetric operads},
Journal = {Journal of Topology}, 
Volume = {11},
Year = {2018},
pages = {559-601}}

\end{document}